\documentclass{article}

\usepackage{geometry}
\usepackage[mathcal]{euscript}
\usepackage{hyperref}
\usepackage{amsmath,amssymb,amsthm}
\usepackage{mathtools}
\usepackage{dutchcal}

\DeclareMathAlphabet{\mathpzc}{OT1}{pzc}{m}{it}

\usepackage{xcolor}
\usepackage{cases}
\newtheorem{thm}{Theorem}[section]
\newtheorem{lemma}[thm]{Lemma}
\newtheorem{prop}[thm]{Proposition}
\newtheorem{cor}[thm]{Corollary}
\theoremstyle{remark}
\newtheorem{rem}[thm]{Remark}

\theoremstyle{definition}
\newtheorem{defn}[thm]{Definition}
\newtheoremstyle{Claim}{}{}{\itshape}{}{\itshape\bfseries}{:}{ }{#1}
\theoremstyle{Claim}

\newcommand{\T}{{\mathbb{T}^d}}

\renewcommand{\H}{\mathcal{H}}

\newcommand{\R}{\mathbb{R}}

\newcommand\sP{\mathpzc{P}}
\newcommand\sQ{\mathpzc{Q}}

\newcommand{\eps}{\varepsilon}
\newcommand{\norm}[1]{\left\lVert#1\right\rVert}
\DeclareMathOperator{\dive}{div}

\def\ds{\displaystyle}

\titlepage
\title{Lipschitz regularity for viscous Hamilton-Jacobi \\ equations with $L^p$ terms}
\author{Marco Cirant and Alessandro Goffi}
\date{\today}

\begin{document}

\maketitle

\begin{abstract} We provide Lipschitz regularity for solutions to viscous time-dependent Hamilton-Jacobi equations with right-hand side belonging to Lebesgue spaces. Our approach is based on a duality method, and relies on the analysis of the regularity of the gradient of solutions to a dual (Fokker-Planck) equation. Here, the regularizing effect is due to the non-degenerate diffusion and coercivity of the Hamiltonian in the gradient variable.
\end{abstract}

\noindent
{\footnotesize \textbf{AMS-Subject Classification}}. {\footnotesize 35F21, 35K55, 35B65}



\section{Introduction}


We study the regularization effect of viscous non-degenerate Hamilton-Jacobi (briefly HJ) equations 
\begin{equation}\label{hjb}
\begin{cases}
\ds \partial_tu(x,t) - \sum_{i,j=1}^d a_{ij}(x,t) \partial_{ij}u(x,t) + H(x, Du(x,t)) = f(x,t) & \text{in $Q_T = \T \times (0, T)$,} \\
u(x, 0) = u_0(x) & \text{in $\T$,}
\end{cases}
\end{equation}
with unbounded right-hand side $f$, continuous initial datum $u_0$, and superlinear Hamiltonian of the model form
\begin{equation}\label{typh}
H(x,p) = h(x)|p|^{\gamma} + b(x) \cdot p, \qquad \gamma > 1, \quad0 < h_0 \le h(x), h, b \in C^1(\T).
\end{equation}
Our main aim is to show that {\it continuous} weak solutions that satisfy the integrability condition
\[
Du \in L^{\sP(\gamma-1)}_{\rm loc}(Q_T) \qquad \text{for $\sP \ge d+2$}
\]
(see in particular Definition \ref{wsol}, {\upshape (i)}) become immediately Lipschitz continuous at positive times. Moreover, we prove that such solutions exist and are unique.

Regularity of solutions to HJ equations has been the object of an extensive literature, and the study of Lipschitz properties has received a particular focus motivated by problems in control theory. The interest has been recently renewed in the study of Mean-Field Game systems, where HJ equations of the form \eqref{hjb} with {\it unbounded} terms $f$ naturally appear. In this setting, few results on Lipschitz regularity are available in the literature, mainly the works by D. Gomes and collaborators \cite{GomesStocEA, GPSMsup}. We also mention some recent results by P. Cardaliaguet, A. Porretta, D. Tonon, and P. Cardaliaguet L. Silvestre \cite{CPT, CSil} on Sobolev and H\"older regularity respectively.

 Note that, depending on the growth of $H$ with respect to the gradient variable, two main regimes are typically identified. If $H$ is sub-quadratic, i.e. $1 < \gamma \leq 2$, then the second order diffusion is the dominating term at small scales. On the other hand, in the super-quadratic case $\gamma > 2$ the diffusion term is considered ``weaker'', and thus typically regarded as a perturbation of a first-order HJ equation. This distinction can be observed heuristically via an $L^\infty$ scaling argument (see e.g. \cite[Section 2.1]{CSil}). A goal of this work is to combine the regularization effects of both the diffusion and the coercivity of the Hamiltonian, and to give a unified treatment of sub- and super-quadratic cases.
 
We give here a very brief review of different techniques that are available to obtain Lipschitz regularity for HJ equations with {\it bounded} (or more regular) $f$, with the understanding that the literature on this subject is too wide to keep track of all the references. First, by means of a classical method by S.V. Bernstein \cite{Bernstein}, Lipschitz regularity for general (sub-quadratic) quasi-linear problems goes back to standard literature \cite{LSU, Serrin}. See also \cite{CLS} for specific results on HJ equations. For sub-quadratic problems, we also mention some techniques based on Sobolev embeddings and interpolation \cite{AC, GPSMsub}, see also \cite{gomesbook}. Then, the so-called Ishii-Lions method \cite{IshiiLions}, has been extensively developed in the realm of HJ equations, see e.g. \cite{BarlesDIeq, BSouga, PP,DirrNguyen}; this method has been coupled succesfully with the Bernstein idea \cite{ArmstrongTran, Davini, LeyNguyen}. Though some of these mentioned results cover the full super-linear regime (if not the sub-quadratic one only), we emphasize that at least continuity of $f$ is always required. \\
Our analysis is based on a duality approach, rather than on viscosity methods. The study of linear equations through their duals (adjoint) is a classical idea, which has been explored recently in the nonlinear framework of HJ equations by L.C. Evans \cite{Ecc}. As already mentioned at the beginning of this introduction, its application to Lipschitz regularity for viscous HJ equations appearing in Mean-Field Game systems has been then investigated in \cite{GomesStocEA, GPSMsup}. In these works some restrictions on the growth of $H$, i.e. $\gamma < 3$ in \cite{GPSMsup}, or on  the space dimension $d$ \cite{GomesStocEA} are imposed. 
Here, we obtain results for all $\gamma > 1$ and $d \in \mathbb N$, by using extensively maximal regularity results in the analysis of the dual equation. We also emphasize that previous works explore a priori regularity of {\it smooth} solutions $u$, while here we deal with least possible {\it weak} solutions to \eqref{hjb}.

We are now ready to state our main results. Assume that $d \ge 2$, and $A=(a_{ij}) :Q_T\to {\rm Sym}(\R^d)$, where ${\rm Sym}(\R^d)$ is the set of symmetric  $d \times d$ real matrices, $a_{ij} \in C(0, T; W^{2,\infty}(\T))$ and
\begin{align}
\tag{A}\label{A1} & \text{ for some }\lambda>0, \quad \lambda|\xi|^2\leq a_{ij}(x,t)\xi_i\xi_j\leq\lambda^{-1}|\xi|^2 \text{ for all }\xi\in\R^d \text{ and all } (x, t) \in Q_T.
\end{align}
Here and in the sequel the summation over repeated indices is understood. We perform our analysis on the flat torus $\T= \R^d / \mathbb Z^d$, to avoid boundary phenomena. The treatment of problems on the whole $\R^d$ should require a modest review of the methods developed here. A local analysis on bounded domains is on the other hand much more delicate (possible boundary blow-up of $Du$ is expected \cite{PSou}), and will be matter of future work. \\
We suppose that $H(x,p)$ is $C^1(\T\times\R^d)$, and
\begin{equation}\label{H}\tag{H}
\begin{aligned}
& \text{there exist constants $\gamma > 1$ and $C_H>0$ such that}\\
& \qquad C_H^{-1}|p|^{\gamma}-C_H\leq H(x,p)\leq C_H(|p|^{\gamma}+1)\ ,\\
& \qquad D_pH(x,p)\cdot p-H(x,p)\geq C^{-1}_H|p|^{\gamma}-C_H\ , \\
& \qquad |D_{x}H(x,p)|\leq C_H(|p|^{\gamma}+1) \ , \\
& \qquad C^{-1}_H|p|^{\gamma-1}- {C}_H \le |D_{p}H(x,p)|\leq C_H|p|^{\gamma-1}+{C}_H \ ,
\end{aligned}
\end{equation}
for every $x\in \T$, $p\in\R^d$. Moreover, we suppose without loss of generality that $H \ge 0$ (if not, one may compensate by adding a positive constant to $f$). Note that our model Hamiltonian \eqref{typh} satisfies \eqref{H}; we mention that the assumptions on $b$ in  \eqref{typh} could be relaxed, but this is beyond the scopes of this paper. Moreover, an explicit dependance with respect to the time variable $t$ could be easily added to $H$ provided that it respects the growth properties stated in \eqref{H}.

The first result concerns the regularizing effect of the equation, namely Lipschitz regularity of {\it weak} solutions $u$ for positive times. If the initial datum is assumed to be Lipschitz, then such regularity can be extended uniformly up to $t = 0$.
Below $\gamma' = \gamma/(\gamma-1)$ is the conjugate exponent of $\gamma$.

\begin{thm}\label{LipReg} Suppose that
\begin{itemize}
\item[$\bullet$] $a_{ij} \in C(0, T; W^{2,\infty}(\T))$ and satisfies \eqref{A1},
\item[$\bullet$] $H \in C^1(\T\times\R^d)$, it is convex in the second variable, and satisfies \eqref{H},
\item[$\bullet$] $f \in L^q(Q_T)$, for some $q > {d+2}$ and $q \ge \frac{d+2}{\gamma'-1}$.
\end{itemize}

{\bf a)} Let $u$ be a {\it local weak} solution to \eqref{hjb} (in the sense of Definition \ref{wsol}) with $\sP = \sQ$ in \eqref{conditionAS}, i.e.
\[
D_p H(\cdot, Du) \in L^{\mathpzc{P}}(\T \times (t_1, T)) \qquad \text{for some} \ \sP \ge d+2 \quad \text{and \quad for all \ $t_1 > 0$}.
\]
Then, $u(\cdot, \tau) \in W^{1,\infty}(\T)$ for all $\tau \in (0, T]$. In particular, for all $t_1 \in (0, T)$ there exists a positive constant $C_1$ depending on $t_1$, $\lambda, \|a\|_{C(W^{2,\infty})}$, $C_H$, $\|u\|_{C(\overline Q_T)}$, $\norm{f}_{L^{q}(Q_T)}$, $q$, $d$, $T$ such that
\begin{equation}\label{locreg}
\|u(\cdot, \tau)\|_{W^{1,\infty}(\T)} \le C_1 \qquad \text{for all} \quad \tau \in [t_1, T].
\end{equation}

{\bf b)} If, in addition, $u_0 \in W^{1,\infty}(\T)$, and $u$ is a {\it global weak} solution, then there exists a positive constant $C_2$ depending on $\lambda, \|a\|_{C(W^{2,\infty})}$, $C_H$, $\|u_0\|_{W^{1,\infty}(\T)}$, $\norm{f}_{L^{q}(Q_T)}$, $q$, $d$, $T$ such that
\begin{equation}\label{globreg}
\|u(\cdot, \tau)\|_{W^{1,\infty}(\T)} \le C_2 \qquad \text{for all} \quad \tau \in [0, T].
\end{equation}

Moreover, the same conclusions hold if $u$ is a weak solution to \eqref{hjb} with $\sP \neq \sQ$ in \eqref{conditionAS} whenever $a_{ij}(x,t) = A_{ij}$ on $Q_T$ for some $A_{ij} \in {\rm Sym}(\R^d)$ satisfying \eqref{A1}.
\end{thm}

Note that if $\gamma \le 2$ (i.e. the sub-quadratic/quadratic regime), then $f$ is required to be in $L^q(Q_T)$ for some $q > {d+2}$, while in the super-quadratic case $\gamma > 2$ conditions on $f$ are more strict.

\medskip

We are then able to show the existence and uniqueness of weak solutions.

\begin{thm}\label{texun} Suppose that the assumptions on $a, f, H$ of Theorem \ref{LipReg} are in force. If $u_0 \in C(\T)$, then there exists a unique local weak solution to \eqref{hjb}.
If $u_0 \in W^{1,\infty}(\T)$, then such a solution is a global weak solution.
\end{thm}

\medskip

Finally, if we assume in addition that $u$ is a {\it classical} solution to \eqref{hjb}, we have the following a priori regularity results. Note that, with respect to the previous Theorem \ref{LipReg}, Lipschitz bounds will depend on weaker informations on the data $a, f$.

\begin{thm}\label{apriori} Suppose that
\begin{itemize}
\item[$\bullet$] $a_{ij} \in C([0, T]; C^{1}(\T))$ and satisfies \eqref{A1},
\item[$\bullet$] $H \in C^2(\T\times\R^d)$ and satisfies \eqref{H},
\item[$\bullet$] $f \in C([0, T]; C^{1}(\T))$,
\item[$\bullet$] $u_0 \in C^1(\T)$.
\end{itemize}
Let 
\begin{equation}\label{qco2}
q > \min\left\{d+2, \frac{d+2}{2(\gamma'-1)}\right\}.
\end{equation}
Then, there exists a positive constant $C_3$ depending on  $q$, $d$, $T$, $\lambda$, $C_H$, $\|u_0\|_{W^{1,\infty}(\T)}$, $\norm{f}_{L^{q}(Q_T)}$, $\|a\|_{C(0, T; W^{1,\infty}(\T))}$, such that every classical solution to \eqref{hjb} satisfies
\begin{equation}\label{locregclass}
\|u(\cdot, \tau)\|_{W^{1,\infty}(\T)} \le C_3 \qquad \text{for all} \quad \tau \in [0, T].
\end{equation}
\end{thm}
Note that \eqref{qco2} reads
\[
q > 
\begin{cases}
d+2 & \text{if $1 < \gamma \le 3$,} \\
 \frac{d+2}{2(\gamma'-1)} & \text{if $\gamma > 3$}.
\end{cases}
\]
 In particular, we obtain ``maximal regularity'' whenever $\gamma \le 3$, that is a control on $\partial_t u, \partial_{ij} u$ and $H(Du)$ in $L^q$ with respect to the the $L^q$ norm of $f$ for any $q > d+2$. The results obtained for $\gamma \geq 3$ are also new, and constitute a first step in the achievement of a parabolic counterpart of a remarkable result by P.-L. Lions \cite[Theorem III.1]{Lions85} in the stationary case, that states Lipschitz (and therefore maximal) regularity of solutions to viscous HJ equations for {\it all} $\gamma > 1$ and $f \in L^p$, $p > d$.
 
\medskip

It is worth remarking that our results apply also to the so-called Kardar-Parisi-Zhang equations
\[
\ds \partial_t v(x,t) - \sum_{i,j=1}^d a_{ij}(x,t) \partial_{ij}v(x,t) = G(x, Dv(x,t)) - f(x,t) \qquad \text{in $Q_T = \T \times (0, T)$,}
\]
whenever $G$ satisfies \eqref{H}. In other words, the sign in front of $H$ (and of $f$) does not matter here. Indeed, it is sufficient to observe that $u(x,t) = -v(x,t)$ solves \eqref{hjb} with $H(x, p) = G(x, -p)$.

\medskip
 
 In the next Section \ref{deriv} we briefly describe our methods, and comment on crucial hypotheses that appear in Theorems \ref{LipReg}, \ref{apriori} and in the Definition \ref{wsol} of weak solutions to \eqref{hjb}. In the rest of Section \ref{basic} we present some preliminary facts and results on the adjoint equation. Sections \ref{slip} and \ref{bernst} will be devoted mainly to the proofs of Theorems \ref{LipReg} and \ref{apriori} respectively, on Lipschitz regularity of solutions. In Section \ref{exun} we will prove the main existence and uniqueness result.

\bigskip

{\bf Acknowledgements.} The authors are members of the Gruppo Nazionale per l'Analisi Matematica, la Probabilit\`a e le loro Applicazioni (GNAMPA) of the Istituto Nazionale di Alta Matematica (INdAM). This work has been partially supported by 
the Fondazione CaRiPaRo
Project ``Nonlinear Partial Differential Equations:
Asymptotic Problems and Mean-Field Games". The authors are indebted to the referees for a careful review which meant a significant improvement of the original version of the manuscript.

\section{Heuristics, functional spaces, weak solutions and basic properties}\label{basic}

\subsection{Heuristic derivation of Lipschitz estimates}\label{deriv}

We begin with a heuristic description of the adjoint method that will be made rigorous in the sequel, and compare with related works \cite{GomesStocEA, GPSMsup}. Let us assume that $u$ is a smooth solution of the viscous HJ equation 
\begin{equation}\label{hjb2}
\partial_t u(x,t) - \Delta u(x,t) + H(Du(x,t)) = f(x,t)
\end{equation}
with $u(\cdot, 0)\in C^1(\T)$. Let $f$ be $C^1$ in the space variable. We differentiate the equation to study the regularity of $Du$, namely, for any direction $\xi\in\R^d$ with $|\xi|=1$, we consider $v=\partial_{\xi} u$. Then, $v$ solves the linearized equation
\begin{equation}\label{veq}
\partial_tv-\Delta v+D_pH(Du)\cdot Dv=\partial_\xi f\ .
\end{equation}
For any $\tau \in (0,T)$, $x_0 \in \T$, we then look at the adjoint equation with singular final datum
\begin{equation}\label{adjoint}
\begin{cases}
-\partial_t\rho-\Delta\rho-\dive(D_pH(Du)\rho)=0&\text{ in } \T \times (0, \tau)\ ,\\
\rho(\tau)=\delta_{x_0}&\text{ on }\T\ .
\end{cases}
\end{equation}
By duality between \eqref{veq} and \eqref{adjoint} we immediately get
\begin{equation*}
\partial_\xi u(x_0, \tau) = \langle v(\tau),\rho(\tau) \rangle=\iint_{\T \times (0, \tau)} \partial_{\xi} f \, \rho + \int_{\T}v\rho(0) = - \iint_{\T \times (0, \tau)} f \, \partial_{\xi} \rho + \int_{\T}\partial_\xi u\,\rho(0).
\end{equation*}
Thanks to integration by parts in the previous formula, we realize that our representation of $\partial_\xi u(x_0, \tau)$ roughly depends on $\|f\|_{L^{q}(Q_T)}$ and $\|D \rho\|_{L^{q'}(Q_T)}$, so, the more we know on the integrability of $D \rho$, the less we can assume on the integrability of the datum $f$. The difficulty here is that $\rho$ depends on $Du$ itself through the divergence term in \eqref{adjoint}, and has a final datum that is a Dirac measure. Note that solutions to heat equations with measure data have gradients in ${L^{q'}(Q_T)}$ just for $q' < (d+2)/(d+1) = (d+2)'$. Therefore, since we do not expect the additional divergence term to improve such regularity, we will always have to require by duality $f$ to be in $L^q$ with $q > d+2$ (which is optimal, see Remark \ref{rem2}).

The transport (divergence) term in \eqref{adjoint} is handled by exploiting a crucial information on the quantity
\begin{equation}\label{energ}
\iint|D_pH(Du)|^{\gamma'}\rho \, dxdt,
\end{equation}
that is obtained using a sort of duality between \eqref{hjb} and \eqref{adjoint}, and has a very precise meaning in terms of optimality in stochastic control problems (see, e.g. \cite{gomesbook} for further discussions). Such a quantity is actually a weighted $L^{\gamma'}(\rho)$ norm of the drift $b = -D_pH(Du)$ that appears in the divergence term, and turns out to be enough to derive bounds for $\|D \rho\|_{L^{q'}(Q_T)}$. This crucial result is stated in Proposition \ref{estFP2} and exploits a delicate combination of maximal parabolic regularity, interpolation and embeddings of parabolic spaces. We emphasize that this key step regarding regularity of the Fokker-Planck equation is carried out in a completely different way with respect to related papers \cite{GomesStocEA, GPSMsup}. In these works, the techniques used to produce estimates on $D \rho$ are expendable under the assumption that $b$ is {\it at least} $L^2(\rho)$, thus limiting the working range of $\gamma$. Here, we have results on the full superlinear range $\gamma > 1$.

In the next sections we make precise all the above formal computations, for more general equations of the form \eqref{hjb}. In the first part of the paper we aim at obtaining Lipschitz regularity of {\it weak solutions} to \eqref{hjb}, in a sense specified below (see Definition \ref{wsol}). The main issues in this program are the following:
\begin{itemize}

  \item To exploit duality between \eqref{hjb} and \eqref{adjoint} in a weak framework, one has to understand the right weak setting for both equations. We realize here that a suitable weak notion guaranteeing Lipschitz regularity for $u$ is basically the usual energy one for both equations (i.e. $u, \rho \in \H^{1}_2$), on any interval $(t_1, T)$, $t_1 > 0$. This relies strongly on the additional assumption $D_pH(Du) \in L^\sQ((t_1,T); L^\sP)$, which can be considered as a requirement for the adjoint equation \eqref{adjoint} rather than for the given HJ equation \eqref{hjb}, but one should always keep in mind the subtle interplay between equations in duality. Of course this forces the final datum $\rho(\tau)$ to be in $L^2$, and therefore introduces an additional approximation step from $L^2$ to $L^1$ in our scheme. Note that energy estimates on $\rho$ are allowed to deteriorate as $t_1 \to 0$: this is to accomodate the lack of global regularity on $[0,T]$ for $u$, that assumes in general the initial datum in the $C^0$ sense only. \\
One may argue that, for $\gamma$ very large, $|Du|^{\gamma-1} \approx D_pH(Du) \in L^\sQ((0,T); L^\sP)$ is very close to $Du \in L^\infty$. We stress in Section \ref{PQq} that to perform this (seemingly) small step, one cannot avoid in general this assumption on $Du$, and therefore our requirements on weak solutions are optimal to guarantee Lipschitz regularity.

  \item A weak solution $u$ is not a priori a.e. differentiable, and $f \in L^q$, so no differentiation procedure of \eqref{hjb} is justified. This is circumvented by considering difference quotients of $u$ in the $x$-variable, which are handled via a method that is again based on the optimality of $-D_pH(Du)$ in stochastic optimal control problems (though here PDE methods will be involved only).  In this step, convexity of $H(x, \cdot)$ plays an important role.

  \end{itemize}

We stress that the study of regularity, rather than the proof of a priori estimates of smooth solutions to \eqref{hjb}, is a key difference with respect to related works \cite{GomesStocEA, GPSMsup} mentioned previously. We take this different viewpoint in the final Section \ref{bernst}: assuming regularity of the solution, we can improve in some directions the previous procedure. First, it is possible to enhance \eqref{energ} by absorbing part of the gradient term in the left hand side of the Lipschitz estimate.
Second, rather than studying the equation for $\partial_\xi u$, we consider the equation for $|Du|^2$, following the classical idea of S.V. Bernstein. This yields a similar ``linearized'' equation, with additional information on $D^2 u$ coming from strict ellipticity of the operator. This allows us to prove {\it a priori} regularity of smooth solutions $u$ to \eqref{hjb} that depend on weaker integrability properties of $f$ and regularity of $a_{ij}$ with respect to $x$.

We finally mention that in a work by A. Porretta \cite{Por15}, the role of the integrability condition $\iint|b|^{\gamma'}\rho \, dxdt < \infty$ in Fokker-Planck equations is explored deeply. Such a condition is indeed proven to guarantee well-posedness of the equation in terms of distributional solutions, provided that $\gamma' \ge 2$, i.e. in the sub-quadratic case, thus showing that the Aronson-Serrin condition on the drift $b \in L^\sQ((t_1,T); L^\sP)$ is not strictly needed. In the work it is also established a kind of duality between Fokker-Planck and HJ equations, in a setting that is much weaker (solutions are unbounded in $W^{1,0}_2(Q_T)$) than the one used here (local boundedness in $W^{1,0}_2(Q_T)$). While the former setting allows for minimal integrability of $f$ and very general existence and uniqueness results to coupled HJ / Fokker-Planck equations, the latter is proven here to produce many additional regularity properties of solutions (in the full range $\gamma > 1$).

\subsection{Functional spaces}

Since we are working in the periodic setting, let us recall that $L^p(\T)$ is the space of all measurable and periodic functions on $\R^d$ belonging to $L^p_{\mathrm{loc}}(\R^d)$, with norm $\|\cdot\|_p=\|\cdot\|_{L^p((0,1)^d)}$. For positive integers $k$, $W^{k,p}(\T)$ is the space of those functions with (distributional) derivatives in $L^p(\T)$ up to order $k$.

For any time interval $(t_1, t_2) \subseteq \R$, let $Q_{t_1, t_2}:=\T\times (t_1, t_2)$. We will also use the notation $Q_{t_2}:= \T\times (0, t_2)$. For any $p\geq1$ and $Q = Q_{t_1, t_2}$, we denote by $W^{2,1}_p(Q)$ the space of functions $u$ such that $\partial_t^{r}D^{\beta}_xu\in L^p(Q)$ for all multi-indices $\beta$ and $r$ such that $|\beta|+2r\leq  2$, endowed with the norm
\begin{equation*}
\norm{u}_{W^{2,1}_p(Q)}=\left(\iint_{Q}\sum_{|\beta|+2r\leq2}|\partial_t^{r}D^{\beta}_x u|^pdxdt\right)^{\frac1p}.
\end{equation*}
The space $W^{1,0}_p(Q)$ is defined similarly, and is endowed with the norm
\[
\norm{u}_{W^{1,0}_p(Q)}:=\norm{u}_{L^p(Q)}+\sum_{|\beta|=1}\norm{D_x^{\beta}u}_{L^p(Q)}\ .
\]
We define the space $\H_p^{1}(Q)$ as the space of functions $u\in W^{1,0}_p(Q)$ with $\partial_tu\in (W^{1,0}_{p'}(Q))'$, equipped with the norm
\begin{equation*}
\norm{u}_{\mathcal{H}_p^{1}(Q)}:=\norm{u}_{W^{1,0}_p(Q)}+\norm{\partial_tu}_{(W^{1,0}_{p'}(Q))'}\ .
\end{equation*}
Denoting by $C([t_1, t_2]; X)$, $C^\alpha([t_1, t_2]; X)$ and $L^q((t_1, t_2); X)$ the usual spaces of continuous, H\"older and Lebesgue functions respectively with values in a Banach space $X$, we have the following isomorphisms: $W^{1,0}_2(Q) \simeq L^2((t_1, t_2); W^{1,2}(\T))$, and
\begin{multline*}
\H_2^1(Q) \simeq \{u \in L^2((t_1, t_2); W^{1,2}(\T)), \, \partial_t u \in  (\, L^2((t_1, t_2); W^{1,2}(\T))\, )'  \} \\ \simeq \big\{u \in L^2((t_1, t_2); W^{1,2}(\T)), \, \partial_t u \in  L^2\big((t_1, t_2); (W^{1,2}(\T))'\big) \big \},
\end{multline*}
and the latter is known to be continuously embedded into $C([t_1, t_2]; L^2(\T))$ (see, e.g., \cite[Theorem XVIII.2.1]{DL}). Sometimes, we will use the compact notation $C(X)$ and $L^q(X)$.

Finally, let $\mathcal P (\T)$ be the space of Borel probability measures on $\T$, endowed with the Kantorovich-Rubinstein distance (which metricizes the weak-* convergence of measures).

\subsection{Weak solutions to viscous HJ equations}

We will require in the sequel $u$ to be a weak (energy) solution in the following sense.

\begin{defn}\label{wsol} We say that
\begin{itemize}
\item[{\it i)}]  $u$ is a {\it local weak} solution to \eqref{hjb} if for all $0 < s < T$
\begin{gather}
\text{$u \in \H_2^1( \T \times (s, T)) \cap C(\overline Q_T)$, \quad $H(\cdot, Du) \in L^1(s,T; L^\sigma(\T))$ for some $\sigma > 1$, }\label{H12} \\
\text{and $D_p H(\cdot, Du) \in L^{\mathpzc{Q}}(s,T; L^{\mathpzc{P}}(\T))$}\label{LqLp} \\
\text{for some $d \le \sP \le \infty$, and $2 \le \sQ \le \infty$ such that $\frac{d}{2\mathpzc{P}} + \frac{1}{\mathpzc{Q}} \le \frac{1}{2}$, }\label{conditionAS}
\end{gather}
and for all $ 0 < s < \tau \le T$, $\varphi \in \H_2^1( \T \times (s, \tau) ) \cap L^\infty(s, \tau; L^{\sigma'}(\T))$
\begin{equation}\label{weake}
\int_s^\tau \langle \partial_t u(t), \varphi(t) \rangle dt +  \iint_{\T \times (s, \tau)}  \partial_i u \, \partial_j(a_{ij} \varphi) + H(x, Du) \varphi \, dxdt =  \iint_{\T \times (s, \tau)} f \varphi \,dxdt
\end{equation}
(here, $\langle \cdot, \cdot \rangle$ denotes the duality pairing between $(W^{1,2}(\T))'$ and $W^{1,2}(\T)$ ).
\item[{\it ii)}]   $u$ is a {\it global weak} solution if \eqref{H12}-\eqref{LqLp}-\eqref{conditionAS} hold for all $0 \le s < T$, that is, on all $Q_T$ (and therefore, \eqref{weake} is also satisfied up to $s = 0$).
\end{itemize}
\end{defn}


\begin{rem}
Under the growth assumptions (H) on the Hamiltonian, one can easily verify the following implications: if $D_p H(x, Du)$ satisfies \eqref{LqLp}-\eqref{conditionAS} for some $\sP = \sQ \ge d+2$, then \eqref{H12} holds for sure whenever $\gamma > \frac{d+2}{d+1}$. Or, if $D_p H(x, Du)$ satisfies \eqref{LqLp}-\eqref{conditionAS} for $\sQ = \infty$ and some $\sP \ge d$, then \eqref{H12} always holds if $\gamma > \frac{d}{d-1}$.
\end{rem}

\subsection{Well-posedness and regularity of the adjoint equation}

This section is devoted to the analysis of the following Fokker-Planck equation
\begin{equation}\label{fplocal}
\begin{cases}
-\partial_t \rho(x,t)-\sum_{i,j=1}^d\partial_{ij}(a_{ij}(x,t)\rho(x,t))+\dive(b(x,t)\, \rho(x,t))=0&\text{ in }Q_\tau\ ,\\
\rho(x,\tau)=\rho_\tau(x)&\text{ in }\T\ .
\end{cases}
\end{equation}
Note that when the vector field $b(x,t) = - D_pH(x, Du(x,t))$, then \eqref{fplocal} becomes the adjoint equation of the linearization of \eqref{hjb}.

Here, $\tau \in (0,T]$, $Q_\tau:=\T\times (0,\tau)$ and $Q_{s,\tau}:= \T \times (s,\tau)$. For $b \in L^\sQ(s,\tau; L^\sP(\T))$ for all $s > 0$, and for some $\sP \ge d$, $\sQ \ge 2$ satisfying \eqref{conditionAS},
a (weak) solution $\rho \in \H_2^1(Q_{s, \tau})$ is such that $\rho(\tau) = \rho_\tau$ in the $L^2$-sense, and
\begin{equation}\label{wfkp}
 -\int_s^\tau \langle \partial_t\rho(t), \varphi(t) \rangle dt + \iint_{Q_{s,\tau}} \partial_{j}(a_{ij}\rho)\partial_i\varphi  - b \rho \cdot D\varphi \,dxdt = 0
\end{equation}
for all $s > 0$ and $\varphi \in \H_2^1(Q_{s,\tau})$. 

Throughout this section we will assume that
\begin{equation}\label{rhoass}
\rho_\tau\in L^\infty(\T), \quad \rho_\tau\geq0 \ \text{ a.e.}, \quad \text{and} \quad \int_\T \rho_\tau(x)\, dx=1.
\end{equation}
Note that $\rho \in C((0, \tau]; L^2(\T))$, so $\rho \in C((0, \tau]; L^1(\T))$, and
\begin{equation}\label{consrho}
\int_\T \rho(x,t)\, dx=1 \qquad \text{for all $t \in (0, \tau]$.}
\end{equation}
This can be easily verified using $\varphi \equiv 1$ as a test function in \eqref{wfkp} and integrating by parts.

\begin{prop}\label{wellADJ}
Let \eqref{A1} be in force, $b \in L^\sQ(s,\tau; L^\sP(\T))$ for all $s > 0$ and for some $\sP \ge d$, $\sQ \ge 2$ satisfying \eqref{conditionAS}, and $\rho_\tau$ be as in \eqref{rhoass}. Then, there exists a weak solution $\rho$ to \eqref{fplocal}. Moreover, $\rho \in L^{\infty}(s,\tau; L^{\sigma'}(\T))$ for all $1 < \sigma' < \infty$ and $s > 0$, and $\rho$ is a. e. non-negative on $Q_\tau$.
\end{prop}

\begin{proof} Existence and regularity of weak solutions to linear equations in divergence form with $b \in L^\sQ(s,\tau; L^\sP(\T))$ is a classical matter that can be found in e.g. \cite{AS, LSU}. Though well known references do not treat directly the periodic setting (but typically the Cauchy-Dirichlet problem), the adaptation of energy methods to $\T$ is straightforward, and can be checked for example following the lines of \cite{BCCS,BOP}. For additional details we refer to \cite{TesiAle}.
\end{proof}

The previous proposition states the well-posedness of the adjoint Fokker-Planck equation for fixed $\rho_\tau \in L^\infty$, and that $\rho(s)$ remains ``almost'' in $L^\infty$ for a.e. $s$. Still, regularity of $\rho$ may deteriorate as $s \to 0$, since the Aronson-Serrin condition on the drift $b$ is not assumed here up to $s = 0$ (see, e.g. \cite[Theorem 4.1]{BCCS}). The main goal is now to derive (weaker) estimates on $\rho$ on the whole $Q_\tau$, that are stable for any $\rho_\tau$ satisfying merely $\|\rho_\tau\|_1 = 1$; one may have in mind that $\rho_\tau$ is an item of a sequence approaching a Dirac delta. These estimates will be unrelated to the Aronson-Serrin condition, and will be achieved using just some information on the integrability of the vector field $b$ with respect to the solution $\rho$ itself, that is a typical datum in the analysis of Hamilton-Jacobi equations.

The following proposition is a modification of \cite[Proposition 2.4]{CT}, and is a kind of parabolic regularity result. The method used here has been inspired by \cite{MPR}, where however estimates are obtained locally in time, and thus are not affected by the regularity of final datum $\rho_\tau$. Similar results for the Sobolev regularity of solutions to Fokker-Planck equations with terminal trace belonging to $L^1$ appeared also in \cite[Proposition 3.10]{Por15} in the sub-quadratic case $\gamma \le 2$, and are compatible with ours.

\begin{prop}\label{estFP}
Let $\rho$ be a (non-negative) weak solution to \eqref{fplocal} and \[1<q'<\frac{d+2}{d+1}.\] Then, there exists $C>0$, depending on $\lambda, \|a\|_{C(W^{1,\infty})},q',d,T$ such that
\begin{equation}\label{estFP1}
\|\rho\|_{\mathcal{H}_{q'}^1(Q_\tau)}\leq C(\|b\rho\|_{L^{q'}(Q_\tau)}+\|\rho\|_{L^{q'}(Q_\tau)}+\|\rho_\tau\|_{L^1(\T)}).
\end{equation}
\end{prop}
Note that $C$ here does not depend on $\tau \in (0, T]$.

\begin{proof} We assume that the coefficients $a_{ij}, b_i$ and $\rho_\tau$ are smooth, and therefore $\rho$ is smooth as well on $Q_\tau$. The general case $D a \in L^\infty(Q_\tau)$, $b$ locally in $L^\sQ(L^\sP(\T))$, $\rho_\tau \in L^\infty$ follows by an approximation argument.

Fix $k=1,...,d$. For $\delta > 0$, let $\psi = \psi_\delta$ be the classical solution to
\begin{equation}\label{auxpsi}
\begin{cases}
\partial_t\psi(x,t)-\sum_{i,j} a_{ij}(x,t)\partial_{ij}\psi(x,t)=(\delta + |\partial_{k}\rho(x,t)|^2)^{\frac{q'-2}{2}}\partial_{k} \rho(x,t)&\text{ in }Q_\tau\ ,\\
\psi(x,0)=0&\text{ on }\T\ .
\end{cases}
\end{equation}
Since $q' < 2$, $\delta > 0$ serves as a regularizing perturbation. By standard parabolic regularity (see Lemma \ref{parabolicreg}), we have (for a positive constant not depending on $\tau \le T$)
\begin{equation}\label{w21q}
\norm{\psi}_{W^{2,1}_{q}(Q_\tau)}\leq C\norm{ (\delta + |\partial_{k}\rho|^2)^{\frac{q'-2}{2}}\partial_{k} \rho }_{L^{q}(Q_\tau)}\leq C\norm{|\partial_{k}\rho|^{q'-1}}_{L^{q}(Q_\tau)}=C\norm{\partial_{k}\rho}^{q'-1}_{L^{q'}(Q_\tau)}.
\end{equation}

Set $\varphi(x,t)=\partial_{k}\psi(x,t)$. Then, $\varphi$ is a classical solution to 
\begin{equation}\label{auxphi}
\begin{cases}
\partial_t\varphi-\sum_{i,j} a_{ij}\partial_{ij}\varphi=\partial_k\left[(\delta + |\partial_{k}\rho|^2)^{\frac{q'-2}{2}}\partial_{k} \rho\right] + \sum_{i,j} \partial_k(a_{ij})\partial_{ij}\psi &\text{ in }Q_\tau\ ,\\
\varphi(x,0)=0&\text{ on }\T\ .
\end{cases}
\end{equation}

Using $\varphi$ as a test function for the equation satisfied by $\rho$,
\begin{equation*}
\iint_{Q_\tau}\rho(\partial_t\varphi-a_{ij}\partial_{ij}\varphi-b \cdot D\varphi) dxdt=\int_{\T}\rho_\tau(x)\varphi(x,\tau)dx\ ,
\end{equation*}
and using the equation in \eqref{auxphi} satisfied by $\varphi$ we get, after integration by parts
\[
 \iint_{Q_\tau} (\delta + |\partial_{k}\rho|^2)^{\frac{q'-2}{2}}|\partial_{k} \rho|^2 -  \partial_k(a_{ij})\partial_{ij}\psi \, \rho + b \rho \cdot D\varphi\,  dxdt= - \int_{\T}\rho_\tau(x)\varphi(x,\tau)dx\ ,
\]
Applying H\"older's inequality,
\begin{multline*}
 \iint_{Q_\tau} (\delta + |\partial_{k}\rho|^2)^{\frac{q'-2}{2}}|\partial_{k} \rho|^2 \,dxdt \le \| Da \|_{L^\infty(Q_\tau)} \| \psi \|_{W^{2,1}_{q}(Q_\tau)} \|\rho\|_{L^{q'}(Q_\tau)} \\ + \|b \rho \|_{L^{q'}(Q_\tau)} \|D\varphi \|_{L^{q}(Q_\tau)} + \| \rho_\tau\|_{L^1(\T)}\|\varphi(\cdot,\tau)\|_{{\infty}}.
\end{multline*}
Since $q>d+2$, by \cite[Lemma II.3.3]{LSU} (see also \cite{CiGo,CT}), the parabolic space $W^{2,1}_{q}(Q_\tau)$ is continuously embedded into $C([0,\tau]; C^1(\T))$, therefore $\|\varphi(\cdot,\tau)\|_{{\infty}} \le \|\psi(\cdot, \tau)\|_{C^1(\T)} \le C\|\psi\|_{W^{2,1}_{q}(Q_\tau)}$ (to be sure that $C$ does not explode as $\tau \to 0$, one has to exploit that $\psi(0)=0$, and argue as in the proof of Proposition \ref{embedding2}). Hence, since $\varphi=\partial_{k}\psi$,
\[
\iint_{Q_\tau} (\delta + |\partial_{k}\rho|^2)^{\frac{q'-2}{2}}|\partial_{k} \rho|^2 \,dxdt \le C( \|\rho\|_{L^{q'}(Q_\tau)} + \|b \rho \|_{L^{q'}(Q_\tau)}+ \| \rho_\tau\|_{L^1(\T)})  \| \psi \|_{W^{2,1}_{q}(Q_\tau)}.
\]
By \eqref{w21q}, letting $\delta \to 0$,
\[
\iint_{Q_\tau}|\partial_{k}\rho|^{q'} \,dxdt \le C( \|\rho\|_{L^{q'}(Q_\tau)} + \|b \rho \|_{L^{q'}(Q_\tau)}+ \| \rho_\tau\|_{L^1(\T)}) \|\partial_{k}\rho\|^{q'-1}_{L^{q'}(Q_\tau)}.
\]
Summarizing, we conclude
\begin{equation}\label{PW}
\norm{D\rho}_{L^{q'}(Q_\tau)}\leq C( \|\rho\|_{L^{q'}(Q_\tau)} + \|b \rho \|_{L^{q'}(Q_\tau)}+ \| \rho_\tau\|_{L^1(\T)})\ .
\end{equation}
By Poincar\'e-Wirtinger inequality, together with the fact that $\int_{\T}\rho(x,t)dx=1$ for all $t \in [0, \tau]$, we obtain
\begin{equation*}
\norm{\rho}^{q'}_{L^{q'}(Q_\tau)}\leq C(\norm{D\rho}^{q'}_{L^{q'}(Q_\tau)} + \tau\|\rho_\tau\|^{q'}_{L^1(\T)})\ ,
\end{equation*}
yielding, together with \eqref{PW}
\[
\norm{\rho}_{W^{1,0}_{q'}(Q_\tau)}\leq C( \|\rho\|_{L^{q'}(Q_\tau)} + \|b \rho \|_{L^{q'}(Q_\tau)}+ \| \rho_\tau\|_{L^1(\T)}).
\]
Finally, for any smooth test function $\varphi$ (which may not vanish at the terminal time $\tau$), again by H\"older's inequality
\begin{multline*}
 \left | \int_0^\tau \langle \partial_t\rho(t), \varphi(t) \rangle dt  \right| \le  \iint_{Q_\tau} |\partial_{j}(a_{ij}\rho)\partial_i\varphi|  + |b \rho |\,| D\varphi|\,dxdt \\
 \le \big[(\|a\|_{L^{\infty}(Q_\tau)} + \|Da\|_{L^{\infty}(Q_\tau)}) \norm{\rho}_{W^{1,0}_{q'}(Q_\tau)} + \|b \rho \|_{L^{q'}(Q_\tau)}\big] \|D \varphi \|_{L^{q}(Q_\tau)}.
\end{multline*}
Thus,
\begin{equation*}
\norm{\partial_t \rho}_{(W^{1,q}(Q_\tau))'}\leq C( \|\rho\|_{L^{q'}(Q_\tau)} + \|b \rho \|_{L^{q'}(Q_\tau)}+ \| \rho_\tau\|_{L^1(\T)})\ .
\end{equation*}
\end{proof}

\begin{prop}\label{estFP2}
Let $\rho$ be the (non-negative) weak solution to \eqref{fplocal} and \[1<q'<\frac{d+2}{d+1}.\] Then, there exists $C>0$, depending on $\lambda, \|a\|_{C(W^{1,\infty})},T,q',d$ such that
\begin{equation}\label{estFP3}
\|\rho\|_{\mathcal{H}_{q'}^1(Q_\tau)}\leq C\left(\iint_{Q_\tau} |b(x,t)|^{r'} \rho(x,t) \, dxdt  + 1\right),
\end{equation}
where
\begin{equation}\label{rdq1}
r' = 1 + \frac{d+2}{q}.
\end{equation}
\end{prop}

\begin{proof}
Inequality \eqref{estFP1}, \eqref{rhoass} and the generalized H\"older's inequality yield
\begin{multline}\label{eqqq0}
\|\rho\|_{\mathcal{H}_{q'}^1(Q_\tau)}\leq C(\|b\rho^{1/r'}\rho^{1/r}\|_{L^{q'}(Q_\tau)}+\|\rho\|_{L^{q'}(Q_\tau)}+1) \\
\le C\left(\left(\iint_{Q_\tau} |b|^{r'} \rho \, dxdt\right)^{1/{r'}} \|\rho\|^{1/r}_{L^{p}(Q_\tau)}+\|\rho\|_{L^{q'}(Q_\tau)}+1\right),
\end{multline}
for $p > q'$ satisfying
\begin{equation}\label{rdq2}
\frac{1}{q'}= \frac{1}{r'} + \frac{1}{rp}.
\end{equation}
Then, by Young's inequality, for all $\eps > 0$
\begin{equation}\label{eqqq1}
\|\rho\|_{\mathcal{H}_{q'}^1(Q_\tau)} \le C\left( \frac{1}{\eps }\iint_{Q_\tau} |b|^{r'} \rho \, dxdt +  \eps\|\rho\|_{L^{p}(Q_\tau)}+\|\rho\|_{L^{q'}(Q_\tau)}+1\right),
\end{equation}
Since $\|\rho\|_{L^1(Q_\tau)} = \tau$, by interpolation between $L^1(Q_\tau)$ and $L^p(Q_\tau)$ we have $\|\rho\|_{L^{q'}(Q_\tau)} \le \tau^{1/{r'}}  \|\rho\|^{1/r}_{L^{p}(Q_\tau)}$, and again by Young's inequality
\begin{equation}\label{eqqq2}
\|\rho\|_{\mathcal{H}_{q'}^1(Q_\tau)} \le \widetilde C\left( \frac{1}{\eps }\iint_{Q_\tau} |b|^{r'} \rho \, dxdt +  \eps\|\rho\|_{L^{p}(Q_\tau)}+1\right),
\end{equation}
One can verify that \eqref{rdq1} and \eqref{rdq2} yield
\[
\frac1p = \frac{1}{q'} - \frac{1}{d+2}.
\]
Indeed, by \eqref{rdq2} we have
\[
\frac1p=\frac{r}{q'}-\frac{r}{r'}=\frac{1}{q'}-\frac{r-1}{q}
\]
and one concludes immediately by using \eqref{rdq1} on the right-hand side of the above equality. The continuous embedding of $\mathcal{H}_{q'}^1(Q_\tau)$ in $L^{p}(Q_\tau)$ stated in Proposition \ref{embedding2} then implies
\[
 \|\rho\|_{L^{p}(Q_\tau)}  \le C_1 \big( \|\rho\|_{\mathcal{H}_{q'}^1(Q_\tau)} + \tau \big)\,.
\]
Hence, the term $ \eps\|\rho\|_{L^{p}(Q_\tau)} $ can be absorbed by the left hand side of \eqref{eqqq2} by choosing $\eps = (2\widetilde C C_1)^{-1}$, thus providing the assertion.
\end{proof}

\section{Lipschitz regularity}\label{slip}

This section is devoted to the proof of Lipschitz regularity of $u$, stated in Theorem \ref{LipReg}. We will assume that the assumptions of Theorem \ref{LipReg} are in force: $a_{ij} \in C(0, T; W^{2,\infty}(\T))$ and satisfies \eqref{A1}, $H \in C^1(\T\times\R^d)$, it is convex in the second variable, and satisfies \eqref{H} and $u_0 \in C(\T)$. Moreover, $f \in L^q(Q_T)$ for some $q > d+2$. 
At a certain stage we will require $q \ge \frac{d+2}{\gamma'-1}$ also.

The result will be obtained using regularity properties of the adjoint variable $\rho$, i.e. the solution to
\begin{equation}\label{fplocaladjoint}
\begin{cases}
\ds -\partial_t\rho(x,t)-\sum_{i,j=1}^d\partial_{ij}(a_{ij}(x,t) \rho(x,t))-\dive\big(D_pH(x,Du(x,t))\, \rho(x,t)\big)=0&\text{ in }Q_\tau\ ,\\
\rho(x,\tau)=\rho_\tau(x)&\text{ on }\T\ ,
\end{cases}
\end{equation}
for $\tau \in (0, T)$, $\rho_\tau\in C^{\infty}(\T)$, $\rho_\tau\geq0$ with $\|\rho_\tau\|_{L^1(\T)}=1$. Recall that $u$ is a weak solution to the viscous Hamilton-Jacobi equation \eqref{hjb}. By the integrability assumptions on $D_p H$, the adjoint state $\rho \in \H_2^1(Q_{s, \tau})$ for all $s > 0$ is, for any $\rho_\tau$, well-defined, non-negative and bounded in $L^\infty(s,\tau; L^{\sigma'}(\T))$ for all $\sigma' > 1$, by a straightforward application of Proposition \ref{wellADJ}.

In what follows, we establish bounds on $\rho$ on the whole $Q_\tau$ that are independent on the choice of $\tau$ and $\rho_\tau \ge 0$ satisfying $\|\rho_\tau\|_{L^1(\T)}=1$.

Before we start, recall that the Lagrangian $L : \T \times \R^d \to \R$, $L(x, \nu) := \sup_{p} \{p \cdot \nu - H(x, p)\}$, namely the Legendre transform of $H$ in the $p$-variable, is well defined by the superlinear character of $H(x, \cdot)$. Moreover, by convexity of $H(x, \cdot)$,
\[
H(x, p) = \sup_{\nu \in \R^d} \{\nu \cdot p - L(x, \nu)\},
\]
and 
\begin{equation}\label{HL}
H(x, p) = \nu \cdot p - L(x, \nu) \quad \text{if and only if} \quad \nu = D_p H(x, p).
\end{equation}
The following properties of $L$ are standard (see, e.g. \cite{CS}): for some $C_L > 0$,
\begin{align}
\tag{L1}\label{L1} & C_L^{-1}|\nu|^{\gamma'} - C_L \le L(x, \nu) \le C_L |\nu|^{\gamma'}\\
\tag{L2}\label{L2} & |D_x L(x, \nu)| \le C_L (|\nu|^{\gamma'} + 1).
\end{align}
for all $\nu \in \R^d$.

\subsection{Estimates on the adjoint variable $\rho$}

Let us point out first that from now on we will denote by $C, C_1, ...$ positive constants that may depend on the data (e.g. $\lambda$, $C_H$, $\norm{u_0}_{C(\T)}$, ...), but do not depend on $\tau$, $\rho_\tau$.

We first start with a duality identity involving $u, \rho$.

\begin{lemma}\label{representationformulaHJ}
Let $u$ be a local weak solution to \eqref{hjb}. Assume that $\rho$ is a weak solution to \eqref{fplocaladjoint}. Then, for all $s \in (0, \tau)$
\begin{equation}\label{reprformula1}
\int_{\T}u(x,\tau)\rho_{\tau}(x)dx=\int_{\T}u(x,s)\rho(x,s)dx +\iint_{Q_{s,\tau}}L(x, D_pH(x,Du))\rho dxdt+\iint_{Q_{s,\tau}} f\rho\,  dxdt.
\end{equation}
Moreover, if $u$ is a global weak solution, the previous identity holds up to $s = 0$.
\end{lemma}
\begin{proof}
Using $-\rho\in \H_2^1(Q_{s,\tau}) \cap L^\infty(s,\tau; L^{\sigma'}(\T))$ as a test function in the weak formulation of problem \eqref{hjb}, $u\in \H_2^1(Q_{s,\tau})$ as a test function for the corresponding adjoint equation \eqref{fplocaladjoint} and summing both expressions, one obtains
\begin{multline*}
-\int_s^\tau \langle \partial_t u(t), \rho(t) \rangle dt -\int_s^\tau \langle \partial_t\rho(t), u(t) \rangle dt \\ + \iint_{Q_{s,\tau}}(D_pH(x,Du)\cdot Du- H(x,Du))\rho dxdt+\iint_{Q_{s,\tau}} f\rho\,  dxdt=0\ .
\end{multline*}
The desired equality follows after integrating by parts in time and using property \eqref{HL} of $L$. Note that since $H(x, Du) \in L^1(s,T; L^\sigma(\T))$, then $L(x, D_pH(Du)) \in L^1(s,T; L^\sigma(\T))$ by \eqref{L1} and \eqref{H}, so all the terms in \eqref{reprformula1} make sense.

The same argument can be used with $s = 0$ in the case that $u$ is a global weak solution.
\end{proof}

We are now ready to prove a crucial estimate on the the integrability of $D_pH$ with respect to $\rho$, that depends in particular on the sup norm $\|u\|_{C(\overline Q_T)}$. Note that this estimate is obtained on the whole parabolic cylinder $Q_\tau$.
\begin{prop}\label{EstimateK}
Let $u$ be a local weak solution to \eqref{hjb} and $\rho$ be a weak solution to \eqref{fplocaladjoint}. Then, there exist positive constants ${C}$ (depending on $\lambda, \|a\|_{C(W^{1,\infty})}$, $\|u\|_{C(\overline Q_T)}$, $C_H$, $\norm{f}_{L^{q}(Q_T)}$, $q, d, T$) such that
\begin{equation}\label{estik}
\iint_{Q_\tau}|D_pH(x,Du(x,t))|^{\gamma'}\rho(x,t) \, dxdt \leq {C} .\
\end{equation}
\end{prop}

\begin{rem} Note that as a straightforward consequence of \eqref{estik}, one has 
\begin{equation}\label{estik1}
\iint_{Q_\tau}|Du(x,t)|^{\beta}\rho(x,t) \, dxdt \leq {C_\beta}\ \qquad \text{for all $1 \le \beta \le \gamma$}.
\end{equation}
Indeed, by \eqref{H}, 
$\iint_{Q_\tau}|Du(x,t)|^{\gamma}\rho(x,t) \, dxdt \leq {C}$, which yields \eqref{estik1} for $\beta = \gamma$. For $\beta < \gamma$ it is sufficient to use Young's inequality and \eqref{consrho}.
\end{rem}

\begin{proof}
Rearrange the representation formula \eqref{reprformula1} to get, for $s\in(0,\tau)$,
\begin{multline}\label{ineqK1}
\iint_{Q_{s,\tau}}L(x, D_pH(x,Du))\rho \, dxdt = \int_{\T}u(x,\tau)\rho_{\tau}(x)dx - \int_{\T}u(x,s)\rho(x,s)dx - \iint_{Q_{s,\tau}} f\rho\,  dxdt.
\end{multline}
Fix some $\eta$ such that $(d+2)/\gamma' < \eta < d+2 \ (< q)$. Use now bounds on the Lagrangian \eqref{L1}, and H\"older's inequality to obtain
\begin{multline}\label{ineqK2}
C_L^{-1} \iint_{Q_{s,\tau}}|D_pH(x,Du)|^{\gamma'}\rho \, dxdt \le \iint_{Q_{s,\tau}}L(x, D_pH(x,Du))\rho \, dxdt \\ \le 2\|u\|_{C(\overline Q_T)} + \|f\|_{L^\eta(Q_{s,\tau})}\|\rho\|_{L^{\eta'}(Q_{s,\tau})} .
\end{multline}

Let now $\bar q$ be such that 
\[
\frac{1}{\eta'} = \frac{1}{\bar q'} - \frac{1}{d+2}
\]
By Proposition \ref{embedding2}, $\mathcal{H}_{\bar q'}^1(Q_{s,\tau})$ is continuously embedded in $L^{\eta'}(Q_{s,\tau})$. Moreover, choosing $\eta>(d+2)/2$ guarantees $\bar q' < (d+2)/(d+1)$, so by inequality \eqref{estFP3} (with $q$ replaced by $\bar q$),
\begin{equation}\label{ineqK3}
\|\rho\|_{L^{\eta'}(Q_{s,\tau})}\leq C( \|\rho\|_{\mathcal{H}_{\bar q'}^1(Q_{s,\tau})}  + 1) \leq C_1\left(\iint_{Q_{s,\tau}} |D_pH(x,Du)|^{r'} \rho(x,t) \, dxdt + 1\right),
\end{equation}
where $r' = 1 + \frac{d+2}{\bar q}$. Plugging this inequality into \eqref{ineqK2}, we obtain
\begin{multline*}
C_L^{-1} \iint_{Q_{s,\tau}}|D_pH(x,Du)|^{\gamma'}\rho \, dxdt \\\leq  2\|u\|_{C(\overline Q_T)} + C_1 \|f\|_{L^\eta(Q_{s,\tau})} \left(\iint_{Q_{s,\tau}} |D_pH(x,Du)|^{r'} \rho(x,t) \, dxdt + 1\right)
\end{multline*}

Finally, the right hand side can be absorbed in the left hand side whenever $r' < \gamma'$ by Young's inequality: this is assured by 
\[
r' = 1 + \frac{d+2}{\bar q} = \frac{d+2}{\eta} < \gamma'.
\]
One then gets \eqref{estik} by taking the limit $s \to 0$ (constants here remain bounded for $s \in (0, \tau)$).

\end{proof}

Integrability of $D_pH$ with respect to $\rho$ provides finally $L^{q'}$ regularity of $D \rho$. From now on, we will suppose that $q > d+2$ and $q \ge \frac{d+2}{\gamma'-1}$. 

\begin{cor}\label{EstimateK1}
Let $u$ be a local weak solution to \eqref{hjb} and $\rho$ be a weak solution to \eqref{fplocaladjoint}. Let $\bar q$ be such that
\[
\bar q > {d+2} \quad \text{and} \quad \bar q \ge \frac{d+2}{\gamma'-1}.
\]
Then, there exists a positive constant ${C}$ such that
\begin{equation*}
\|\rho\|_{\mathcal{H}_{\bar q'}^1(Q_\tau)}\leq C\ ,
\end{equation*}
where ${C}$ depends in particular on $\lambda, \|a\|_{C(W^{1,\infty})}, C_H$, $\|u\|_{C(\overline Q_T)}$, $\norm{f}_{L^{q}(Q_\tau)}$, $\bar q , d, T$ (but not on $\tau, \rho_\tau$).
\end{cor}

\begin{proof} Since $\bar q' < \frac{d+2}{d+1}$, \eqref{estFP3} applies (with $q = \bar q$), yielding
\[
\|\rho\|_{\mathcal{H}_{\bar q'}^1(Q_\tau)}\leq C\left( \iint_{Q_\tau} |D_pH(x,Du(x,t))|^{r'} \rho(x,t) \, dxdt + 1\right),
\]
with 
\[
r' = 1 + \frac{d+2}{\bar q} \le \gamma'.
\]
If $r'=\gamma'$, use Proposition \ref{EstimateK} to conclude. Otherwise, if $r' < \gamma'$,
use Young's inequality first to control $\iint |D_pH(x,Du(x,t))|^{r'} \rho \, dxdt$ with $\iint|D_pH(x,Du)|^{\gamma'}\rho \, dxdt + \tau$.
\end{proof}
\begin{rem}\label{remregrho} It is worth noting that in the sub-quadratic regime $\gamma \le 2$, the information $b \in L^{\gamma'}(\rho)$ is strong enough to guarantee $\|D \rho\|_{L^{q'}(Q_T)}$ for all $q' < (d+2)/(d+1)$, that is expected for distributional solutions to heat equations with $L^1$ data (see e.g. \cite{Por15}). We can then regard the ${\rm div}()$ term in \eqref{adjoint} as perturbation of a heat equation. On the other hand, in the super-quadratic case $\gamma > 2$, we are just able to prove the weaker regularity $\|D \rho\|_{L^{q'}(Q_T)}$ for $q' \le q_\gamma'$, with $q_\gamma' < (d+2)/(d+1)$, where actually $q_\gamma' \to 1$ as $\gamma \to \infty$. As expected, in the super-quadratic case the Hamiltonian term in \eqref{hjb} may overcome the regularizing effect of Laplacian.
\end{rem}

Finally, if one thinks  $\rho(t)$ as a flow of probability measures, then $\rho$ enjoys also some H\"older regularity in time.

\begin{cor}\label{EstimateK3}
Let $u$ be a local weak solution to \eqref{hjb} and $\rho$ be a weak solution to \eqref{fplocaladjoint}. Then, there exists a positive constant ${C}$ such that
\begin{equation*}
{\bf d_1}(\rho(t), \rho(t'))\leq C|t-t'|^{\frac12 \wedge \frac1\gamma}\qquad \forall t,t' \in [0,\tau] ,
\end{equation*}
where ${C}$ depends in particular on $\lambda, \|a\|_{C(W^{1,\infty})}$, $C_H$, $\|u\|_{C(\overline Q_T)}$, $\norm{f}_{L^{q}(Q_\tau)}$, $d$, $T$ (but not on $\tau, \rho_\tau$).
\end{cor}

\begin{proof}
Since $\rho$ solves the Fokker-Planck equation \eqref{fplocaladjoint} with drift $D_pH(x, Du(x,t))$, given the $L^1$ bound \eqref{estik} on $|D_p H(\cdot, Du)|^{\gamma'} \rho$, the result is a straightforward application of \cite[Lemma 4.1]{CGPT}.
\end{proof}

\subsection{Further bounds for global weak solutions}

If $u$ is a global weak solution, i.e. an energy solution up to initial time, it is possible to control its $\sup$ norm in terms of $\|u_0\|_{C(\T)}$. This will be done in the next proposition.

\begin{prop}\label{stimeu}
There exists $C>0$ (depending on $\lambda, \|a\|_{C(W^{1,\infty})},T,d$) such that any global weak solution $u$ to \eqref{hjb} satisfies
\begin{equation}\label{uinfty}
\|u(\cdot, \tau)\|_{C(\T)} \le C \qquad \text{for all $\tau \in [0, T]$}.
\end{equation}
\end{prop}

\begin{proof}
First, we prove a bound from above for $u$:
\begin{equation}\label{esttestheat}
u(x,\tau)\leq \|u_0\|_{C(\T)} + C\|f\|_{L^q(Q_\tau)}
\end{equation}
for all $\tau \in (0, T)$ and $x \in \T$. Consider indeed the (strong) non-negative solution of the following backward problem
\begin{equation*}
\begin{cases}
-\partial_t\mu(x,t)-\sum_{i,j}\partial_{ij}(a_{ij}(x,t)\mu(x,t))=0&\text{ on }Q_{\tau}\ ,\\
\mu(x,\tau)=\mu_{\tau}(x)&\text{ on }\T\ .
\end{cases}
\end{equation*}
with $\mu_{\tau}\in C^{\infty}(\T)$, $\mu_\tau\geq0$ and $\norm{\mu_{\tau}}_{L^1(\T)}=1$. Note that $\mu$ is a solution of a Fokker-Planck equation of the form \eqref{fplocal} with drift $b \equiv 0$. Then, since $q' < (d+2)/(d+1)$, by Proposition \ref{estFP2} there exists a positive constant $C$ (not depending on $\tau, \mu_\tau$) such that $\|\mu\|_{\mathcal{H}_{q'}^1(Q_\tau)}\leq C$.

Use $\mu$ as a test function in the weak formulation of the Hamilton-Jacobi equation \eqref{hjb} to get
\begin{equation*}
\int_{\T}u(x,\tau)\mu_{\tau}(x)dx=\int_{\T}u_0(x)\mu(x,0)dx+\iint_{Q_\tau}f\mu dxdt-\iint_{Q_\tau}H(x,Du)\mu dxdt\ .
\end{equation*}
Applying H\"older's inequality to the second term of the right-hand side of the above inequality and the fact that $\|\mu(t)\|_{L^1(\T)}=1$ for all $t\in[0,\tau]$, we get
\begin{equation*}
\int_{\T}u(x,0)\mu(x,0)dx + \int_{0}^{\tau}\int_{\T}f\mu dxdt\leq \|u_0\|_{C(\T)} + C\|f\|_{L^q(Q_\tau)}\ ,
\end{equation*}
 By the assumption $H(x,p)\geq 0$, we then conclude
\begin{equation*}
\int_{\T}u(x,\tau)\mu_{\tau}(x)dx\leq  \|u_0\|_{C(\T)} + C\|f\|_{L^q(Q_\tau)}\ .
\end{equation*}
Finally, by passing to the supremum over $\mu_\tau\geq0$, $\norm{\mu_{\tau}}_{L^1(\T)}=1$, one deduces the estimate \eqref{esttestheat} by duality.

To prove the bound from below of $u$, one can argue exactly as in the proof of Proposition \ref{EstimateK}, starting from the representation formula \eqref{ineqK1} with $s = 0$. Using the additional upper bound \eqref{esttestheat},
\[
\iint_{Q_\tau}|D_pH(x,Du(x,t))|^{\gamma'}\rho(x,t) \, dxdt \leq 2\|u_0\|_{C(\T)} + C\|f\|_{L^q(Q_\tau)} +  \|f\|_{L^\eta(Q_{\tau})}\|\rho\|_{L^{\eta'}(Q_{\tau})}.
\]
This provides as before a control on $\iint_{Q_\tau}|D_pH(x,Du))|^{\gamma'}\rho \, dxdt$ and thus on $\|\rho\|_{L^{\eta'}(Q_{\tau})}$, which depends on $\|u_0\|_{C(\T)}$ instead of the full sup norm $\|u\|_{C(\overline Q_T)}$. Going back to \eqref{reprformula1},
\[
\int_{\T}u(x,\tau)\rho_{\tau}(x)dx \ge \int_{\T}u(x,0)\rho(x,0) - C_L \iint_{Q_\tau} \rho(x,t) dxdt+\iint_{Q_\tau} f\rho\,  dxdt.
\]
Since $\iint f\rho$ can be bounded (from below) by H\"older's inequality, 
\[
\int_{\T}u(x,\tau)\rho_{\tau}(x)dx \ge - \|u(\cdot,0)\|_{C(\T)} - C_L  \tau - C.
\]
Since $\rho_\tau$ can be arbitrarily chosen so that $\| \rho_\tau \|_{L^1(\T)} = 1$, we have the desired result.
\end{proof}

\subsection{Proof of Theorem \ref{LipReg}}

The following theorem contains the core argument of Lipschitz regularity.

\begin{thm}\label{StimaLip} Let $u$ be a local weak solution to \eqref{hjb}. Suppose first that \eqref{conditionAS} holds with $\sP = \sQ$.

Let $\eta \in C_0^\infty((0, T])$ be a smooth function satisfying $0 \le \eta(t) \le 1$ for all $t$. Then, $(\eta u)(\cdot, \tau) \in W^{1,\infty}(\T)$ for all $\tau \in [0,T]$, and there exists $C > 0$ depending on 
$\lambda, \|a\|_{C(W^{1,\infty})}$, $\|D^2 a\|_{L^\infty(Q_\tau)}$, $C_H$, $\|u\|_{C(\overline Q_T)}$, $\norm{f}_{L^{q}(Q_T)}$, $q , d, T$ such that
\[
\eta(\tau) \|D u(\cdot, \tau)\|_{L^{\infty}(\T)} \le C \Big( 
\|D a\|_{L^\infty(Q_\tau)} \|\eta Du\|_{L^{(d+2)(\gamma-1)}(Q_\tau)} +\sup_{(0, T)}|\eta'(t)| + 1 \Big)
\]
for all $\tau \in [0,T]$.

Without requiring $\sP = \sQ$ in \eqref{conditionAS}, but assuming in addition that $D a \equiv 0$ on $Q_T$, we have the same assertion, and in particular
\[
\eta(\tau) \|D u(\cdot, \tau)\|_{L^{\infty}(\T)} \le C \Big( 
\sup_{(0, T)}|\eta'(t)| + 1 \Big)
\]
for all $\tau \in [0,T]$.
\end{thm}

\begin{proof}
\textbf{Step 1.} Since $H$ is convex and superlinear we can write for a.e. $(x,t) \in Q_T$
\begin{equation*}
H(x,Du(x, t))=\sup_{\nu \in \R^d}\{\nu\cdot Du(x, t)-L(x,\nu)\}.
\end{equation*}
Hence we get, for $0 < s < \tau$,
\begin{multline}\label{HJoptimalLOC}
\int_s^\tau \langle \partial_t u(t), \varphi(t) \rangle dt +  \iint_{Q_{s,\tau}}  \partial_i u(x, t) \, \partial_j(a_{ij}(x, t) \varphi(x, t)) + [\Xi(x, t) \cdot Du(x, t)-L(x,\Xi(x, t))] \varphi \, dxdt \\
\le  \iint_{Q_{s,\tau}} f(x, t) \varphi(x, t) \,dxdt
\end{multline}
for all test functions $\varphi \in  \H_2^1(Q_{s,\tau}) \cap L^\infty(s,\tau;L^{\sigma'}(\T))$ and measurable $\Xi : Q_{s,\tau} \to \R^d$ such that $L(\cdot, \Xi(\cdot, \cdot)) \in L^1(s,\tau;L^{\sigma}(\T))$ and $\Xi\cdot Du \in  L^1(s,\tau;L^{\sigma}(\T))$. Note that the previous inequality becomes an equality if $\Xi(x,t) = D_pH(x,Du(x,t))$ in $Q_{s,\tau}$.

We fix $\rho_\tau$ as in \eqref{rhoass}. Set \[w(x,t)=\eta(t) u(x,t).\] 
 
 Use now \eqref{HJoptimalLOC} with $\Xi(x,t) = D_pH(x,Du(x,t))$ and $ \varphi = \eta \rho \in  \H_2^1(Q_\tau) \cap L^\infty(s,\tau;L^{\sigma'}(\T))$ for all $1<\sigma'<\infty$, where $\rho$ is the adjoint variable (i.e. the weak solution to \eqref{fplocaladjoint}) to find
\begin{multline}\label{loc1}
\int_s^\tau \langle \partial_t w(t), \rho(t) \rangle dt +  \iint_{Q_{s,\tau}}  \partial_i w \, \partial_j(a_{ij} \rho) + D_pH(x,Du) \cdot Dw \rho -L(x,D_pH(x,Du)) \eta \rho \, dxdt \\
=  \iint_{Q_{s,\tau}} f \eta \rho \,dxdt +  \iint_{Q_{s,\tau}} u \eta' \rho \,dxdt.
\end{multline}
Then, use $w \in \H_2^1(Q_T)$ as a test function in the weak formulation of the equation satisfied by $\rho$ to get
\begin{equation}\label{inte1}
-\int_s^\tau \langle \partial_t\rho(t), w(t) \rangle dt + \iint_{Q_{s,\tau}} \partial_{j}(a_{ij}\rho)\partial_i w  + D_pH(x,Du) \rho \cdot Dw \,dxdt = 0.
\end{equation}
We now fix $s$ small so that $\eta(s) = 0$. We then obtain, subtracting the previous equality to \eqref{loc1}, and integrating by parts in time
\begin{multline}\label{optloc}
\int_{\T}w(x,\tau)\rho_{\tau}(x)dx= \iint_{Q_{s,\tau}}\eta(t)f(x,t)\rho(x,t)dxdt\\
+\iint_{Q_{s,\tau}}\eta(t)L\big(x,D_pH(x,Du(x,t))\big)\rho(x,t)dxdt+\iint_{Q_{s,\tau}} \eta'(t)u(x,t)\rho(x,t)dxdt.
\end{multline}

For $h>0$ and $\xi\in\R^d$, $|\xi|=1$, define $\hat{\rho}(x,t):=\rho(x-h\xi,t)$. After a change of variables in \eqref{fplocaladjoint}, it can be seen that $\hat{\rho}$ satisfies, using $w$ as a test function,
\begin{multline}\label{inte2}
-\int_s^\tau \langle \partial_t \hat \rho(t), w(t) \rangle dt \\
+ \iint_{Q_{s,\tau}} \partial_{j}\big(a_{ij}(x-h\xi ,t)\hat \rho(x,t)\big)\partial_i w  +  D_pH(x-h\xi,Du(x-h\xi,t))\hat \rho(x,t) \cdot Dw(x,t) \,dxdt = 0.
\end{multline}
As before, plugging $\Xi(x,t) = D_pH(x-h\xi,Du(x-h\xi,t))$ and $\varphi = \eta \hat \rho$ in \eqref{HJoptimalLOC} yields
\begin{multline*}
\int_s^\tau \langle \partial_t w(t), \hat \rho(t) \rangle dt + \\ \iint_{Q_{s,\tau}}  \partial_i w \, \partial_j(a_{ij} \hat \rho) + D_pH(x-h\xi,Du(x-h\xi,t)) \cdot Dw \hat \rho -L(x,D_pH(x-h\xi,Du(x-h\xi,t))) \eta \hat \rho \, dxdt \\
\le  \iint_{Q_{s,\tau}} f \eta \hat \rho \,dxdt +  \iint_{Q_{s,\tau}} u \eta' \hat \rho \,dxdt.
\end{multline*}
Hence, subtracting \eqref{inte2} to the previous inequality, 
\begin{multline*}
\int_{\T}w(x,\tau)\hat{\rho}_{\tau}(x)dx \le 
\iint_{Q_{s,\tau}} \partial_{j}\Big(\big(a_{ij}(x-h\xi,t) - a_{ij}(x,t)\big) \hat \rho(x,t)\Big)\partial_i w \,dxdt\\
+ \iint_{Q_{s,\tau}} L(x,D_pH(x-h\xi,Du(x-h\xi,t))) \eta \hat \rho \, dxdt+ \iint_{Q_{s,\tau}} f \eta \hat \rho \,dxdt +  \iint_{Q_{s,\tau}} u \eta' \hat \rho \,dxdt,
\end{multline*}
which, after the change of variables $x \mapsto x + h\xi$, becomes
\begin{multline}\label{suboptloc}
\int_{\T}w(x + h\xi,\tau){\rho}_{\tau}(x)dx \le 
\iint_{Q_{s,\tau}} \partial_{j}\Big(\big(a_{ij}(x-h\xi,t) - a_{ij}(x,t)\big) \rho(x,t)\Big)\partial_i w \,dxdt\\
+ \iint_{Q_{s,\tau}} \eta(t) L(x+h\xi,D_pH(x,Du(x,t)))\rho(x,t) \, dxdt \\+ \iint_{Q_{s,\tau}} f \eta \hat \rho \,dxdt +  \iint_{Q_{s,\tau}} u \eta' \hat \rho \,dxdt,
\end{multline}
Taking the difference between \eqref{suboptloc} and \eqref{optloc} we obtain
\begin{equation}\label{step1}
\begin{aligned}
&\int_{\T}(w(x + h\xi, \tau)-  w(x,\tau)) {\rho}_{\tau}(x)dx \le 
\iint_{Q_{s,\tau}} \partial_{j}\Big(\big(a_{ij}(x-h\xi,t) - a_{ij}(x,t)\big) \rho(x,t)\Big)\partial_i w \,dxdt\\
&\qquad + \iint_{Q_{s,\tau}} \eta(t) \Big(L(x+h\xi,D_pH(x,Du(x,t))) - L(x,D_pH(x,Du(x,t)))\Big)\rho(x,t) \, dxdt \\
&\qquad + \iint_{Q_{s,\tau}} \eta(t) f(x,t) \big(\rho(x-h\xi, t)- \rho(x,t) \big) \,dxdt \\
&\qquad +  \iint_{Q_{s,\tau}} \eta'(t) u(x,t)\big( \rho(x-h\xi,t) - \rho(x,t) \big) \,dxdt.
\end{aligned}
\end{equation}

\textbf{Step 2.} We now estimate all the right hand side terms of \eqref{step1}. We stress that constants $C, C_1, \ldots$ are not going to depend on $\tau, \rho_\tau, h, \xi$. 

Regarding the first term, assuming that $\sP = \sQ$ holds in \eqref{conditionAS}, we have by the growth assumptions \eqref{H} on $D_pH$ that $\eta Du \in L^{(d+2)(\gamma-1)}(Q_\tau)$. Note that this fact will be used in the next chain of inequalities only. By Young's and Holder's inequality
\begin{multline}\label{aterm}
\left| \iint_{Q_{s,\tau}} \partial_{j}\Big(\big(a_{ij}(x-h\xi,t) - a_{ij}(x,t)\big) \rho(x,t)\Big)\partial_i w \,dxdt \right| = 
\\
\left|\iint_{Q_{s,\tau}} \big(\partial_{j}a_{ij}(x-h\xi,t) - \partial_{j}a_{ij}(x,t)\big) \rho\, \partial_i w \,dxdt + \iint_{Q_{s,\tau}} (a_{ij}(x-h\xi,t) - a_{ij}(x,t)) \partial_j\rho \, \partial_i w \,dxdt \right| \\
\le \|D^2 a\|_{L^\infty(Q_{s,\tau})} |h| \iint_{Q_{s,\tau}} |Du|\rho \, dxdt +  \|D a\|_{L^\infty(Q_{s,\tau})} |h|  \iint_{Q_{s,\tau}} |\eta Du|\, |D\rho| \, dxdt \\
\le C|h|\left(\iint_{Q_{s,\tau}} |Du|^\gamma \rho \, dxdt + \tau +  \|D a\|_{L^\infty(Q_{s,\tau})} \|\eta Du\|_{L^{(d+2)(\gamma-1)}(Q_{s,\tau})} \|D\rho\|_{L^{(\, (d+2)(\gamma-1) \, )'}(Q_{s,\tau})} \right) \\
\le C_1|h|\left(  \|D a\|_{L^\infty(Q_{s,\tau})} \|\eta Du\|_{L^{(d+2)(\gamma-1)}(Q_{s,\tau})} +1 \right),
\end{multline}
where in the last inequality we used \eqref{estik1} and Corollary \ref{EstimateK1} (with $\bar q = (d+2)(\gamma-1) = (d+2)/(\gamma'-1) \ $ ).

Next, using first the mean value theorem (that yields a function $\zeta : \T \to \T$), then property \eqref{L2} of $D_xL$ and \eqref{estik},
\begin{multline*}
\left| \iint_{Q_{s,\tau}} \eta(t) \Big(L(x+h\xi,D_pH(x,Du(x,t))) - L(x,D_pH(x,Du(x,t)))\Big)\rho(x,t) \, dxdt \right|  \\
\le |h| \iint_{Q_{s,\tau}}  \big|D_x L\big(\zeta(x), D_pH(x,Du(x,t))\big)\big| \rho(x,t) \, dxdt  \\
\le C_L |h| \iint_{Q_{s,\tau}} \big(|D_pH(x,Du(x,t))|^{\gamma'} +1\big) \rho(x,t) \, dxdt \le C |h|.
\end{multline*}
Denote by $D^h \rho(x,t):=|h|^{-1}(\rho(x+h\xi,t)-\rho(x,t))$. Then, for the term involving $f$ we use again Corollary \ref{EstimateK1}, with $\bar q = q$, and control the $L^{q'}$ norm of difference quotient $D^h \rho$ via $D\rho$ (as in, e.g. \cite[Theorem 2.1.6]{Z}), to get
\begin{multline*}
\left|  \iint_{Q_{s,\tau}} \eta(t) f(x,t) \big(\rho(x-h\xi, t)- \rho(x,t) \big) \,dxdt \right |   \\
\le |h| \iint_{Q_{s,\tau}} |f(x,t)|\,|D^h \rho(x,t)| \,dxdt \le |h| \|f\|_{L^q(Q_{s,\tau})} \|D\rho\|_{L^{q'}(Q_{s,\tau})} \le C|h|.
\end{multline*}
Finally, by boundedness of $u$ stated in \eqref{uinfty} and again Corollary \ref{EstimateK1}
\begin{multline*}
\left|  \iint_{Q_{s,\tau}} \eta'(t) u(x,t)\big( \rho(x-h\xi,t) - \rho(x,t) \big) \,dxdt \right| \le |h| \big(\sup_{(0, T)}|\eta'(t)| \big) \|u\|_{L^\infty(Q_{s,\tau})} \|D\rho\|_{L^{1}(Q_{s,\tau})} \\
\le C|h| \sup_{(0, T)}|\eta'(t)|.
\end{multline*}
Plugging all the estimates in \eqref{step1} we obtain
\begin{equation}\label{step2}
  \int_{\T}(w(x + h\xi, \tau)-  w(x,\tau)) {\rho}_{\tau}(x)dx   \\ \le C|h| \Big(
  \|D a\|_{L^\infty(Q_\tau)} \|\eta Du\|_{L^{(d+2)(\gamma-1)}(Q_\tau)} +\sup_{(0, T)}|\eta'(t)| + 1 \Big).
\end{equation}

\textbf{Step 3.} Since \eqref{step2} holds for all smooth $\rho_\tau \ge 0$ with $\|\rho_\tau\|_{L^1(\T)} = 1$, we get
\begin{equation*}
\eta(\tau) [u(x + h\xi, \tau)-  u(x,\tau)] \le C|h| \Big(
 \|D a\|_{L^\infty(Q_\tau)} \|\eta Du\|_{L^{(d+2)(\gamma-1)}(Q_\tau)} +\sup_{(0, T)}|\eta'(t)| + 1 \Big)
\end{equation*}
for all $x \in \T$, $\xi \in \R^d$, $h > 0$. Thus, $u(\cdot, \tau)$ is Lipschitz continuous, and
\begin{equation*}
\eta(\tau) \|D u(\cdot, \tau)\|_{L^{\infty}(\T)} \le C \Big(
 \|D a\|_{L^\infty(Q_\tau)} \|\eta Du\|_{L^{(d+2)(\gamma-1)}(Q_\tau)} +\sup_{(0, T)}|\eta'(t)| + 1 \Big).
\end{equation*}
Since $C$ does not depend on $\tau \in (0,T)$, we have proved the theorem.

Finally, for the special case $D a \equiv 0$ on $Q_T$, one may follow the very same lines, with the difference that there is no need to control the term appearing in \eqref{aterm} (which is identically zero). Therefore, there is no need to keep track of $\|\eta Du\|_{L^{(d+2)(\gamma-1)}(Q_\tau)}$, and therefore the theorem is proven without assuming the constraint $\sP = \sQ$ in \eqref{conditionAS}.

\end{proof}

The following lemma shows that $\|Du\|_{L^{\gamma}(Q_T)}$ can be bounded by a constant depending on the data only.

\begin{lemma}\label{StimaDu}
Let $u$ be a local weak solution. Then, there exists a constant $C$ depending on $C_H$, $\|u\|_{C(\overline Q_T)}$, $\norm{f}_{L^{q}(Q_T)}$, $\|Da\|_{L^\infty(Q_\tau)}$, $q, d, T$ such that 
\[
\|Du\|_{L^{\gamma}(Q_T)}\leq C.
\]
\end{lemma}
\begin{proof}
Plugging $\varphi \equiv 1$ as a test function in the weak formulation of \eqref{hjb} we obtain, for $s > 0$,
\[
\int_{\T}u(x, T) \,dx - \int_{\T}u(x, s) \,dx +  \iint_{Q_{s, T}}  \partial_i u \, \partial_j(a_{ij}) + H(x, Du) \, dxdt =  \iint_{Q_{s, T}} f dxdt
\]
Hence, using \eqref{H}, and Young's inequality we get
\begin{multline*}
C_H\iint_{Q_{s, T}}|Du|^{\gamma}\,dxdt\leq \|u(\cdot, T) \|_{C(\T)} + \|u(\cdot, s) \|_{C(\T)} + \frac{C_H}{2}\iint_{Q_{s, T}}|Du|^{\gamma}\,dxdt \\ + C T \|Da_{ij}\|^{\gamma'}_{L^\infty(Q_{s, T})} + \iint_{Q_{s, T}} |f|^q dxdt + T + C_H^{-1} T.
\end{multline*}
Therefore, we conclude by passing to the limit $s \to 0$.

\end{proof}

We are now ready to prove the main theorem on Lipschitz regularity stated in the introduction.

\begin{proof}[Proof of Theorem \ref{LipReg}] For $t_1 \in (0,T)$, let $\eta=\eta(t)$ be a non negative smooth function on $[0,T]$ satisfying $\eta(t) \le 1$ for all $t$, $\eta(t) \equiv 1$ on $[t_1, T]$ and vanishing on $[0, t_1/2]$. Then, Theorem \ref{StimaLip} yields $u(\cdot, \tau) \in W^{1,\infty}(\T)$ for all $\tau \in (0,T)$, and the existence of $C > 0$ (depending on the data and $\eta$, so $t_1$ itself) such that
\[
\eta(\tau) \| D u(\cdot, \tau)\|_{L^{\infty}(\T)} \le C \big( \|D a\|_{L^\infty(Q_\tau)} \|\eta Du\|_{L^{(d+2)(\gamma-1)}(Q_\tau)} + 1 \big)
\]
for all $\tau \in [0,T]$. If $(d+2)(\gamma-1) \le \gamma$, we immediately conclude \eqref{locreg} using Lemma \ref{StimaDu}. Otherwise, by interpolation of $L^{(d+2)(\gamma-1)}(Q_\tau)$ between $L^\gamma(Q_\tau)$ and $L^\infty(Q_\tau)$ we get
\[
\eta(\tau) \| Du(\cdot, \tau)\|_{L^{\infty}(\T)} \le C \left( \|D a\|_{L^\infty(Q_\tau)} \|\eta Du\|_{L^\infty(Q_\tau)}^{1-\frac{\gamma}{(d+2)(\gamma-1)}}\|\eta Du\|_{L^\gamma(Q_\tau)}^{\frac{\gamma}{(d+2)(\gamma-1)}} + 1 \right),
\]
that implies  \eqref{locreg} after passing to the supremum with respect to $\tau \in (0, T)$, and again using Lemma \ref{StimaDu} to control $\|\eta Du\|_{L^\gamma(Q_\tau)}$.

\medskip

To prove the global in time bound \eqref{globreg} one may follow the same lines, using $\eta \equiv 1$ on $[0, T]$ instead. Being the solution global, $s = 0$ can indeed be chosen throughout the proof of Theorem \ref{StimaLip}, and norms $\|u\|_{C(\overline Q_T)}$ can be replaced by $\|u_0\|_{C(\T)}$ in view of Proposition \ref{stimeu}. Note that an additional term $\int_{\T}(u(x+h,0)-u(x,0)){\rho}(x,0)dx $ pops up in \eqref{step1}: this can be easily bounded by $\|D u_0\|_{L^\infty(\T)}$.

\medskip

Finally, if $a_{ij}(x,t) = A_{ij}$ on $Q_T$ for some $A_{ij}$ satisfying \eqref{A1}, then $D a \equiv 0$ on $Q_T$, and we obtain the same conclusion, exploiting the fact that Theorem \ref{StimaLip} does not require anymore $\sP = \sQ$.

\end{proof}

\subsection{Beyond Lipschitz regularity}

Once Lipschitz regularity is established, one can deduce further properties of weak solutions. Indeed, the viscous HJ equation \eqref{hjb} can be treated in terms of regularity as a linear equation, being the $H(x, Du)$ term (locally in time) bounded in $L^\infty$. Thus, the classical Calder\'on-Zygmund parabolic theory applies, and the so-called maximal regularity for $u$ follows, i.e.: $\partial_t u, \partial_{ij} u, H(x, Du) \in L^q$.

\begin{cor} Under the assumptions of Theorem \ref{LipReg}, any local weak solution $u$ of \eqref{hjb} is a strong solution belonging to $W^{2,1}_q(\T\times(t_1, T))$ for all $t_1 \in (0,T)$, namely it solves \eqref{hjb} almost everywhere in $Q_T$.
\end{cor}

\begin{proof} For any $t_1 > 0$, Theorem \ref{LipReg} yields $H(x, Du(x,t)) \in L^\infty(\T\times(t_1/2, T))$. Therefore, since $f \in L^q(\T\times(t_1/2, T))$ and $q > d+2$, there exists a weak (energy) solution $v$ to the linear equation
\begin{equation}\label{lineq}
\ds \partial_tv(x,t) - \sum_{i,j=1}^d a_{ij}(x,t) \partial_{ij}v(x,t) = - H(x, Du(x,t)) + f(x,t) \qquad \text{$\in L^q(\T\times(t_1/2, T))$,} 
\end{equation}
that satisfies $v(t_1/2) = u(t_1/2)$ in the $L^2$-sense, and enjoys the additional strong regularity property $W^{2,1}_q(\T\times(t_1, T))$. This can be proven using, e.g., local estimates in \cite[Theorem IV.10.1]{LSU}. Since weak solutions to \eqref{lineq} are unique, $u$ coincides a.e. with $v$ on $\T\times(t_1, T)$, and we obtain the assertion.
\end{proof}


%

\subsection{Some remarks on the exponents $\sP$, $\sQ$, $q$}\label{PQq}


In the following remarks, we stress the importance of the condition $D_p H \in L^{\mathpzc{Q}}(L^{\mathpzc{P}}(\T))$ with $\sP$, $\sQ$ satisfying
\begin{equation}\label{as2}
\frac{d}{2\mathpzc{P}} + \frac{1}{\mathpzc{Q}}  \le \frac{1}{2}.
\end{equation}
Not only it guarantees Lipschitz regularity of $u$, but is also related to uniqueness of solutions in the distributional sense. In the following examples it is indeed possible to observe multiplicity of solutions; among them, there is one that is a local weak, Lipschitz continuous solution, while the other(s) are not, showing therefore that Lipschitz regularity for positive times stated in Theorem \ref{LipReg} fails in general without extra integrability properties of $D_pH(x,Du)$.
 \\
We will also comment on  the condition $f \in L^q(Q_T)$, $q > d+2$.

\begin{rem}\label{rem1} We consider first the super-quadratic regime $\gamma > 2$. For $\mathpzc{Q} = \infty$, \eqref{as2} reads
\[
D_p H(x, Du) \in L^{\infty}(L^{\sP}(\T)) \qquad \text{ for some  } \quad \sP\ge d.
\]
Let $a_{ij}=\delta_{ij}$ and $H(x, p) = |p|^\gamma$, $\gamma > 2$. For $c, \alpha > 0$, we consider the (time-independent) function
\[
u_1(x,t)=c\psi(x)|x|^{\alpha}\qquad \text{on $Q_T$,}
\]
where $\psi$ is a smooth function having support on $B_{1/2}(0)$ and is identically one in $B_{1/4}(0)$. Note that $\psi$ has the role of a localizing term only, so that $u_1(x,t)$ is a representative on $[-1/2,1/2]^d$ of a periodic function on $\R^d$. If we let
\[
\qquad \alpha=\frac{\gamma-2}{\gamma-1}, \qquad c = \frac{(d+\alpha-2)^{\frac{1}{\gamma-1}}}{\alpha}
\]
then $u_1$ solves, for some $f_1 \in L^\infty(\T)$ (that vanishes on $B_{1/4}(0)$)
\begin{equation}\label{patologic}
\begin{cases}
\ds \partial_tu - \Delta u(x,t) + |Du(x,t)|^\gamma = f_1(x) \\
u(x, 0) = c\psi(x)|x|^{\alpha},
\end{cases}
\end{equation}
in the sense that it satisfies all the requirements in Definition \ref{wsol}, except the Aronson-Serrin condition \eqref{LqLp}-\eqref{conditionAS}. More precisely,
\[
(\gamma-1)|Du|^{\gamma-1} = |D_p H(x, Du)| \in L^{\infty}(0,T;L^{\sP}(\T)) \qquad \text{ if and only if  } \quad \sP< d.
\]
Moreover, $u_1(\cdot, \tau)$ is clearly not Lipschitz continuous for any $\tau \in [0, T]$. 

Note that $u(x, 0)  \in C(\T)$ and $f_1 \in L^\infty(Q_T)$, so by Theorem \ref{texun} there exists a unique solution to \eqref{patologic} in the sense of Definition \ref{wsol}. Thus, \eqref{patologic} admits two distinct strong solutions, but only the one satisfying fully the Definition \ref{wsol}, in particular the crucial integrability condition on $D_p H(x, Du)$, enjoys Lipschitz regularity.
\end{rem}

\begin{rem}\label{rem1bis} In the sub-quadratic regime $1 + 2/(d+2) < \gamma < 2$, for $a_{ij}=\delta_{ij}$ and $H(x, p) = |p|^\gamma$, we can produce an energy solution to \eqref{hjb} such that $D_p H(x, Du) \in L^{\mathpzc{Q}}(0,T; L^{\mathpzc{P}}(\T))$ if and only if 
\[
\frac{d}{2\mathpzc{P}} + \frac{1}{\mathpzc{Q}}  > \frac{1}{2},
\]
that is not Lipschitz continuous, and not even bounded in $L^\infty$ uniformly on $\overline Q_T$. It then satisfies all requirements of Definition \ref{wsol} except the Aronson-Serrin condition \eqref{LqLp}-\eqref{conditionAS} and the continuity up to $t = 0$: the initial datum is assumed in the $L^2$-sense only.

The construction of such a $u$ is based on the existence, for $k > 0$ small, of $U \in C^2(0,\infty) \cap C^1[0,\infty)$ to the Cauchy problem
\[
\begin{cases}
U''(y) + \left(\frac{d-1}y + \frac y 2 \right)U'(y) + U(y) + |U'(y)|^\gamma = 0 & \text{for $0 < y < \infty$} \\
U'(0) = 0 \\
U(0) = \alpha_0,
\end{cases}
\]
for some $\alpha_0 > 0$, that satisfies for some positive $c$
\[
|U(y)| + |U'(y)| + |U''(y)| \le c e^{-y} \qquad \text{as $y \to \infty$}. 
\]
The existence of  such a $U$ is proven in \cite[Section 3]{BaSW}, see in particular Theorem 3.5, Proposition 3.11, Proposition 3.14 and Remark 3.8 (see also \cite{LM17}). As in our Remark \ref{rem1}, we need a smooth localization term $\psi$ having support on $(-1/2, 1/2)$ and identically one in $[-1/4, 1/4]$. Let then 
\[
u_2(x,t)=-t^{-\sigma}U(|x|\,t^{-1/2})\psi(|x|),\qquad \sigma = \frac{2-\gamma}{2(\gamma-1)}.
\]
We have that $u_2$ is a classical solution to 
\begin{equation}\label{hj5}
\ds \partial_tu(x,t) - \Delta u(x,t) + |Du(x,t)|^\gamma = f_2(x,t),
\end{equation}
where $u_2(0) = 0$ in the $L^2$-sense (since $\gamma > 1 + 2/(d+2)$). Moreover,
\begin{align*}
& f_2(x, t) =
-t^{-\sigma - 1}\bigg\{ \bigg[U''(|x|\,t^{-1/2}) + \left(\frac{d-1}{|x|\,t^{-1/2}} + \frac {|x|\,t^{-1/2}} 2 \right)U'(|x|\,t^{-1/2}) + kU(|x|\,t^{-1/2})\bigg]\psi(|x|)  \\
& \qquad\qquad\qquad + \Big|U'(|x|\,t^{-1/2})\psi(|x|)  + t^{1/2}U(|x|\,t^{-1/2}) \psi'(|x|)\Big|^\gamma \\
& \qquad\qquad\qquad + 2t^{1/2}U'(|x|\,t^{-1/2}) \psi'(|x|) + tU(|x|\,t^{-1/2}) \psi''(|x|)+ \frac{d-1}{|x|}t U(|x|\,t^{-1/2}) \psi'(|x|) \bigg\}.
\end{align*}
Note that $f_2(x, t)$ is identically zero on $|x| \le 1/4$ and $|x| \ge 1/2$; otherwise, it is bounded in $L^\infty$, since $|U(|x|\,t^{-1/2})| + |U'(|x|\,t^{-1/2})| + |U''(|x|\,t^{-1/2})| \le c e^{-t^{-1/2}/4}$. Therefore, one should expect the existence of a weak solution to the HJ equation \eqref{hj5} with zero initial datum that is Lipschitz continuous on the whole $Q_T$ (by Theorem \ref{texun}), but such a solution cannot be $u_2$, since $u_2(t)$ becomes unbounded as $t \to 0$.

\end{rem} 

\begin{rem}\label{rem2} To have Lipschitz bounds for solutions to \eqref{hjb}, one cannot avoid in general the condition
\begin{equation}\label{qd2}
f \in L^q(Q_T) \quad \text{for some} \quad q > d+2.
\end{equation}
This constraint is actually imposed by the linear (heat) part of \eqref{hjb}. Consider indeed $a_{ij}=\delta_{ij}$ and $H(x, p) = |p|^\gamma$, $\gamma > 1$. For $T>0$, let $\chi \in C^\infty_0(\R^d)$, $\Gamma(x,t)$ be fundamental solution of the heat equation in $\R^d$, $f_3(x,t) := \chi(x/\sqrt{T-t})[\sqrt{T-t} \, \log(T-t)]^{-1}$ and $u_3$ be the function
\[
u_3(x,t) := \iint_{\R^d \times (0,t)} f_3(y,s) \Gamma(x-y,t-s) \, dyds \qquad \text{on $Q_T$}
\]
Clearly, $u_3$ is a classical solution to
\[
\begin{cases}
\ds \partial_tu(x,t) - \Delta u(x,t) + |Du(x,t)|^\gamma = f_3(x,t) + |Du_3(x,t)|^\gamma \\
u(x, 0) = 0,
\end{cases}
\]
$f_3 \in L^q(Q_T)$ for all $q \le d + 2$ and $|Du_3|^\gamma \in L^\infty(0, T; L^\beta(\T))$ for all $\beta < \infty$. In turn, we have that  $\|D u_3(\cdot, t)\|_\infty \to \infty$ as $t \to T$. Note that this example can be recast into the periodic setting by multiplying $u_3$ by a cut-off function $\psi$, as in the previous remarks.

Therefore, with respect to integrability requirements on $f$, Theorem \ref{apriori} is optimal, at least when $\gamma < 3$, namely when $d+2 \ge \frac{d+2}{2(\gamma'-1)}$. We do not know whether \eqref{qd2} is enough also when $\gamma \ge 3$.
\end{rem}

\section{Existence and uniqueness of solutions}\label{exun}

This section is devoted to the proof of existence and uniqueness of solutions to the HJ equation \eqref{hjb}. 

\begin{proof}[Proof of Theorem \ref{texun}]
{\bf Existence. }We start with a sequence of classical solutions $u_n$ to regularized problems
\[
\begin{cases}
\ds \partial_tu_n(x,t) - \sum_{i,j=1}^d a_{ij}(x,t) \partial_{ij}u_n(x,t) + H(x, Du_n(x,t)) = f_n(x,t) & \text{in $Q_T = \T \times (0, T)$,} \\
u_n(x, 0) = u_{n,0}(x) & \text{in $\T$,}
\end{cases}
\]
where $f_n, u_{n,0}$ are smooth functions converging to $f, u_{0}$ in $L^q(Q_T), C(\T)$ respectively. The existence of solutions to the regularized equations can be proven using standard methods, as detailed in \cite{TesiAle} (see also \cite{CiGo}).

The global bound on $\|u_n\|_{C(\overline Q_T)}$ depending on $\|u_0\|_{C(\T)}$ (see Proposition \ref{stimeu}) and the local in time Lipschitz estimate \eqref{locreg} hold, namely, for any fixed $t_1 > 0$,
\[
\|D u_n(\cdot, t)\|_{L^\infty(\T)} \le C_{t_1} \qquad \text{for all} \quad t \in [t_1, T].
\]
Hence, since $f_n$ is equibounded in $L^q(Q_T)$, $u_n$ is equibounded in $W^{2,1}_q(Q_{t_1, T})$ by standard maximal parabolic regularity (e.g. \cite[Theorem IV.10.1]{LSU}). Then, weak limits $\partial_t u, D^2 u$ exist (up to subsequences), and are in $L^q$ locally in time. Moreover, since $q > d+2$, parabolic embeddings of $W^{2,1}_q$ (see e.g. \cite{CiGo, LSU,TesiAle}) guarantee that $u_n$ and $D u_n$ are equibounded and equicontinuous in $[t_1, T]$ for all $t_1 > 0$. Therefore, Ascoli theorem and a further diagonalization argument imply that, again up to subsequences, $u_n$ converges uniformly on $[t_1, T]$ for all $t_1 > 0$ to some limit $u$, and the same convergence holds for $Du_n$. Note that the desired limit equation is locally satisfied in the strong sense, namely a.e. on $Q_T$. 

To prove that $u$ is a local weak solution, it just remains to show that it is continuous up to $t = 0$. This is a delicate step since the control on $Du$ deteriorates as $t \to 0$. We start with the l.s.c. inquality
\[
u_0(x_0) \le \liminf_{ \substack{ x \to x_0 \\ t \to 0 } } u(x, t) \qquad \forall x_0.
\]
The following fact will be crucial: for all $(\bar x, \bar t) \in Q_T$, there exists $\rho = \rho_{\bar x, \bar t} \in C^{\frac12 \wedge \frac1\gamma}([0, \bar t], \mathcal P(\T)) \cap \mathcal{H}_{q'}^1(Q_{\bar t})$ such that $\rho_{\bar x, \bar t}(\bar t) = \delta_{\bar x}$ and 
\begin{equation}\label{uuscin}
u(\bar x, \bar t) \ge \int_\T u_0(x)  \rho_{\bar x, \bar t}(0, dx) + \iint_{Q_{\bar t}} f(x,t) \rho_{\bar x, \bar t}(x,t) dx dt - C_L \bar t,
\end{equation}
and $\rho_{\bar x, \bar t}$ is bounded in $C^{\frac12 \wedge \frac1\gamma}([0, \bar t], \mathcal P(\T)) \cap \mathcal{H}_{q'}^1(Q_{\bar t})$ uniformly in $(\bar x, \bar t)$. Indeed, let $u_n$ be as in the previous part of the proof, and $\rho_n$ be the corresponding adjoint variable solving \eqref{fplocal}, where $\rho_n(\bar t)$ is any sequence converging to $\delta_{\bar x}$ in the sense of measures. By duality (see Lemma \ref{representationformulaHJ}) we get
\[
\int_{\T}u_n(x,\bar t)\rho_n(x,\bar t)=\int_{\T}u_{n,0}(x)\rho_n(x,0)dx +\iint_{Q_{\bar t}}\Big(L(x, D_pH(x,Du_n))\rho_n dxdt + f_n \rho_n \Big)\,  dxdt.
\]
Moreover, $\rho_n$ is bounded in $C^{\frac12 \wedge \frac1\gamma}([0, \bar t], \mathcal P(\T)) \cap \mathcal{H}_{q'}^1(Q_{\bar t})$ by means of Corollaries \ref{EstimateK1} and \ref{EstimateK3}, and these bounds do not depend on $\rho_n(\bar t)$ nor on $(\bar x, \bar t)$. By \eqref{L1}, $L \ge -C_L$. Moreover, $u_{n,0}(\cdot)$ and $u_n(\cdot,\bar t)$ converge uniformly in $\T$, $\rho_n(t)$ converges in the sense of measures, $f_n$ converges strongly to $f$ in $L^q(Q_{\bar t})$ while $\rho_n$ enjoys the same convergence in the weak $L^{q'}$ sense, eventually up to subsequences (actually it could be made strong convergence by compact embeddings of parabolic spaces). Hence we obtain \eqref{uuscin} by passing to the limit $n \to \infty$.

Fix now $x_0 \in \T$, and let $(\bar x_m, \bar t_m)$ be any sequence such that $(\bar x_m, \bar t_m) \to (x_0, 0)$. By adding $u_0(x_0)$ to both sides of \eqref{uuscin}, rearranging the terms and using H\"older's inequality, we have
\begin{multline*}
u_0(x_0) \le u(\bar x_m, \bar t_m) \\ + \|f\|_{L^q(\T \times (0, \bar t_m))} \| \rho_{\bar x_m, \bar t_m} \|_{L^{q'}(\T \times (0, \bar t_m))} + C_L \bar t_m + \int_\T u_0(x) \big(\delta_{x_0} - \rho_{\bar x_m, \bar t_m}(0) \big)(dx).
\end{multline*}
On one hand, $\|f\|_{L^q(\T \times (0, \bar t_m))}\to 0$ as $\bar t_m \to 0$, while $ \| \rho_{\bar x_m, \bar t_m} \|_{L^{q'}}$ is equibounded. On the other hand, as $\bar x_m \to x_0$,
\[
\int_\T u_0(x) \big(\delta_{x_0} - \rho_{\bar x_m, \bar t_m}(0) \big)(dx) = u_0(x_0) - u_0(\bar x_m) + \int_\T u_0(x) \big(\rho_{\bar x_m, \bar t_m}(\bar t_m) - \rho_{\bar x_m, \bar t_m}(0) \big)(dx) \to 0,
\]
by continuity of $u_0$, and the fact that ${\bf d_1}(\rho_{\bar x_m, \bar t_m}(\bar t_m),  \rho_{\bar x_m, \bar t_m}(0))\leq C|\bar t_m|^{\frac12 \wedge \frac1\gamma} \to 0$ implies the convergence of $\rho_{\bar x_m, \bar t_m}(\bar t_m)$ to $\rho_{\bar x_m, \bar t_m}(0)$ in the weak sense of measures. We then get the claimed lower semicontinuity of $u$ on $\overline Q_T$.

The reverse inequality 
\[
u_0(x_0) \ge \limsup_{ \substack{ x \to x_0 \\ t \to 0 } } u(x, t) \qquad \forall x_0
\]
can be obtained following analogous lines: instead of testing the approximating equation for $u_n$ by solutions $\rho_n$ to the adjoint Fokker-Planck equation, it is sufficient to use 
\[
-\partial_t\mu_n(x,t)-\sum_{i,j}\partial_{ij}\big(a_{ij}(x,t)\mu_n(x,t)\big)=0 \qquad \text{ on }Q_{\bar t}\ ,\\
\]
i.e. a solution of a Fokker-Planck equation of the form \eqref{fplocal} with drift $b \equiv 0$, such that $\mu_n(\bar t)$ converges to $\delta_{\bar x}$ in the sense of measures. By duality with $u_n$ and $H \ge -C_H$, it holds
\begin{equation*}
\int_{\T}u_n(x,\bar t)\mu_{n}(x, \bar t)dx \le \int_{\T}u_0(x)\mu_n(x,0)dx+\iint_{Q_{\bar t}}f_n \mu_n dxdt + C_H \bar t, 
\end{equation*}
and by taking limits
\begin{equation*}
u(\bar x, \bar t) \le \int_\T u_0(x)  \mu_{\bar x, \bar t}(0, dx) +\iint_{Q_{\bar t}}f \mu dxdt + C_H \bar t,
\end{equation*}
so it is possible to proceed as before.

%

\medskip

{\bf Uniqueness.} Consider two solutions $u_1, u_2$ of the HJ equation, and take their difference $w := u_1 - u_2$ on $\overline Q_T$. Let $\tau \in (0, T]$. By convexity of $H(x, \cdot)$, $w$ solves
\[
\int_s^\tau \langle \partial_t w(t), \varphi(t) \rangle dt +  \iint_{  \T\times(s, \tau)}  \partial_i w \, \partial_j(a_{ij} \varphi) + D_p H(x, Du_2) \cdot Dw \, \varphi \, dxdt \le  0
\]
for all $s \in (0, \tau)$, and $w(\cdot, 0) = 0$. Let now $\rho$ be adjoint variable with respect to $u_2$, namely $\rho$ be the weak solution to 
\begin{equation}\label{fkpu}
\begin{cases}
\ds -\partial_t\rho(x,t)-\sum_{i,j=1}^d\partial_{ij}(a_{ij}(x,t) \rho(x,t))-\dive\big(D_pH(x,Du_2(x,t))\, \rho(x,t)\big)=0&\text{ in }Q_\tau\ ,\\
\rho(x,\tau)=\rho_\tau(x)&\text{ on }\T\ ,
\end{cases}
\end{equation}
for some non-negative and smooth probability density $\rho_\tau$. Then, by duality we get
\[
\int_\T w(x,\tau) \rho_\tau(x) dx \le \int_\T w(x,s) \rho(x,s) dx.
\]
Since $w \in C(\overline Q_T)$, it is uniformly continuous on $\overline Q_T$, so $w(\cdot, t) \to w(\cdot, 0) \equiv 0$ uniformly in $\T$. Moreover, $\int_\T w(x,s) \rho(x,s) dx = \int_\T [w(x,s) - w(x,0)] \rho(x,s) dx$. Thus, by H\"older's inequality and $\|\rho(s)\|_{L^1(\T)} = 1$, $\int_\T w(s) \rho(s) \to 0$, yielding
\[
\int_\T w(x,\tau) \rho_\tau(x) dx \le 0
\]
for arbitrary $\rho_\tau$. As $\rho_\tau$ varies, $u_1(\tau) \le u_2(\tau)$ follows, and by exchanging the role of $u_1$ and $u_2$ and varying $\tau$, we eventually obtain $u_1 \equiv u_2$.

\medskip

{\bf Additional regularity. } When $u_0 \in W^{1,\infty}$, using global Lipschitz bounds \eqref{globreg} one can bring Lipschitz (and further) regularity of $u_n$ to the limit solution $u$ on the whole time interval $[0, T]$.

\end{proof}

\begin{rem} Note that the uniqueness proof works in the sub-quadratic case $\gamma \le 2$ if one requires $u_0 \in L^\infty(\T)$ and $u_i (s) \stackrel{\ast}{\rightharpoonup} u_0$ in $L^\infty$ only. This follows by the fact that $\rho$ in \eqref{fkpu} can be proven (as in Proposition \ref{EstimateK}) to satisfy $\int_0^T \int |D_pH(x,Du_2)|^{\gamma'} \rho < \infty$. When $\gamma' \ge 2$, then $\rho \in C([0,T], L^1(\T))$ by \cite[Theorem 3.6]{Por15}. Strong convergence of $\rho(s)$ in $L^1$ and weak-* convergence of $u_1(s) - u_2(s)$ is then enough to have  $\int_\T w(s_n) \rho(s_n) \to 0$ along some sequence $s_n \to 0$. We believe that existence and Lipschitz regularity of solutions could be addressed in this weaker framework, but this is a bit beyond the scopes of this paper. Nevertheless, these considerations are in line with the principle that in the super-quadratic case $\gamma > 2$, the HJ equation ``sees points'' \cite{CSil}, and thus requires $u_0$ to be continuous in order to be well-posed, while for $\gamma \le 2$ it may be enough to have informations a.e. at initial time. 
\end{rem}

\section{A priori estimates: Bernstein's and the adjoint methods}\label{bernst}

This section is devoted to the proof of Theorem \ref{apriori}, and complements regularity results of the previous section. Here, $u$ is a classical solution to \eqref{hjb}. This will allow to perform the Bernstein's method, namely to analyse the equation satisfied by $|Du|^2$. The adjoint of such an equation is basically \eqref{fplocaladjoint}. As before we will exploit the interplay between the equation itself and its adjoint.

We will assume that $a_{ij} \in C([0, T]; C^{1}(\T))$ and satisfies \eqref{A1}, $H \in C^2(\T\times\R^d)$ and satisfies \eqref{H}, $f \in C([0, T]; C^{1}(\T))$, $u_0 \in C^1(\T)$ and 
\[ q > \min\left\{d+2, \frac{d+2}{2(\gamma'-1)}\right\}.
\]
As before, for any fixed $\tau \in (0, T)$, $\rho_\tau\in C^{\infty}(\T)$, $\rho_\tau\geq0$ with $\|\rho_\tau\|_{L^1(\T)}=1$, let $\rho$ be the (classical) solution to \eqref{fplocaladjoint}. Note that Proposition \ref{stimeu}, Lemma \ref{representationformulaHJ} and Proposition \ref{EstimateK} apply. We start with a revised version of Corollary \ref{EstimateK1}. 

\begin{cor}\label{EstimateK2}
Let $u$ and $\rho$ be (classical) solutions to \eqref{hjb} and \eqref{fplocaladjoint} respectively. Let $\bar q$ be such that
\[
\bar q >  \frac{d+2}{2(\gamma'-1)}.
\]
Then, there exist constants $C > 0$ and $0<\delta<1$ such that
\begin{equation*}
\|\rho\|_{\mathcal{H}_{\bar q'}^1(Q_\tau)}\leq C\big( \|Du\|^{1-\delta}_{L^\infty(Q_\tau)} + 1\big)\ ,
\end{equation*}
where ${C}$ depends in particular on $\lambda, \|a\|_{C(W^{1,\infty})}$, $C_H$, $\norm{u_0}_{C(\T)}$, $\norm{f}_{L^{q}(Q_T)}$, $\bar q , d, T$ (but not on $\tau, \rho_\tau$).
\end{cor}

A straightforward consequence of the corollary is that
\begin{equation}\label{estik4}
\|\rho\|_{L^{\bar p'}(Q_\tau)}\leq C\big( \|Du\|^{1-\delta}_{L^\infty(Q_\tau)} + 1\big), \qquad \text{for all} \quad \bar p >  \frac{d+2}{2(\gamma'-1) + 1}.
\end{equation}
Indeed, since $\bar{q}'<\frac{d+2}{d+1}$, Proposition \ref{embedding2} gives the result.

\begin{proof} Since $\bar q' < \frac{d+2}{d+1}$, \eqref{estFP3} applies (with $q = \bar q$), yielding by \eqref{H}
\begin{multline*}
\|\rho\|_{\mathcal{H}_{\bar q'}^1(Q_\tau)}\leq C\left(\iint_{Q_\tau} |D_pH(x,Du)|^{r'} \rho \, dxdt  + 1\right) \\
\leq C_1\left(\iint_{Q_\tau} |Du|^{(\gamma-1)r'} \rho\, dxdt  + 1\right)\\ \le C_1\left(\|Du\|^{1-\delta}_{L^\infty(Q_\tau)} \iint_{Q_\tau} |Du|^{(\gamma-1)r' - 1 + \delta} \rho\, dxdt  + 1\right),
\end{multline*}
with $r' = 1 + (d+2)\bar{q}^{-1}$. Note that $\delta > 0$ can be chosen small so that $(\gamma-1)r' - 1 + \delta \le \gamma$. One then uses the estimate \eqref{estik1} on $\iint |Du|^\gamma \rho$ (and Proposition \ref{stimeu}) to conclude.
\end{proof}

We are now ready to prove our main a priori Lipschitz regularity result.

\begin{proof}[Proof of Theorem \ref{apriori}] {\bf Step 1.} Set $z(x,t):=\frac{|Du(x,t)|^2}{2}$ on $Q_T$. Straightforward computations yield
\[
\partial_i z=Du\cdot D\partial_i u\ , \quad
\partial_{ij}z=D\partial_ju\cdot D\partial_iu+Du\cdot D\partial_{ij}u\ , \quad
\partial_tz=Du\cdot D(\partial_t u)\ ,
\]
which give
\begin{equation}\label{Trz}
\mathrm{Tr}(AD^2z)=\sum_{k=1}^d AD\partial_ku\cdot D\partial_ku+Du\cdot D\{\mathrm{Tr}(AD^2u)\}-\sum_{k=1}^d\partial_ku\mathrm{Tr}(\partial_kAD^2u)\ .
\end{equation}
Then, differentiating the HJ equation \eqref{hjb} with respect to $x_k$, multiplying the resulting equation by $u_{x_k}$, and summing for $k=1,\ldots,d$, one finds
\[
Du\cdot D(\partial_t u)-Du\cdot D\{\mathrm{Tr}(AD^2u)\}+D_pH\cdot Dz+D_xH\cdot Du=Df\cdot Du\ .
\]
Therefore, by plugging \eqref{Trz} into the previous equality we obtain the following equation satisfied by $z$
\begin{equation}\label{eqz}
\partial_tz-\mathrm{Tr}(AD^2z)+\sum_{k=1}^d AD\partial_ku\cdot D\partial_ku+D_pH\cdot Dz\\
=\sum_{k=1}^d\partial_ku\mathrm{Tr}(\partial_kAD^2u)-D_xH\cdot Du+Df\cdot Du\ .
\end{equation}
Using the uniform ellipticity condition \eqref{A1} we estimate the third term on the left-hand side by
\[
\sum_{k=1}^d AD\partial_ku\cdot D\partial_ku\geq \lambda \mathrm{Tr}((D^2u)^2).
\]
Multiply \eqref{eqz} by the adjoint variable $\rho$ and integrate by parts in space-time to get
\begin{multline}\label{ineqBern}
\int_{\T}z(x,\tau)\rho_{\tau}(x)\,dx+\lambda\iint_{Q_\tau}\mathrm{Tr}((D^2u)^2)\rho\,dxdt\leq \int_\T z(x,0)\rho(x,0)\,dxdt + \\ \iint_{Q_\tau} |D_xH||Du|\rho\,dxdt
+\iint_{Q_\tau}Df\cdot Du\rho\,dxdt+\iint_{Q_\tau} \partial_ku\mathrm{Tr}(\partial_kAD^2u)\rho\,dxdt.
\end{multline}
{\bf Step 2.} We proceed by estimating the four terms on the right hand side of \eqref{ineqBern}. First,
\begin{equation}\label{bc1}
\int_\T z(x,0)\rho(x,0)\,dxdt \le \frac12 \|D u(\cdot, 0)\|_{L^\infty(\T)}^2.
\end{equation}

Second, thanks to \eqref{H}, Proposition \ref{EstimateK} and Young's inequality,
\begin{equation}\label{bc2}
\iint_{Q_\tau} |D_xH||Du|\rho \leq \|Du\|_{L^\infty(Q_\tau)}\left[C_H \iint_{Q_\tau}|Du|^{\gamma}\rho\,dxdt+C_H \tau\right]
\leq C_2+\frac18\|Du\|_{L^\infty(Q_\tau)}^2.
\end{equation}

We then consider $\iint Df\cdot Du\rho$. Integrating by parts,
\begin{multline*}
\left|\iint_{Q_\tau} Df\cdot Du\rho\,dxdt\right|=\left|\iint_{Q_\tau} f\mathrm{div}(Du\rho)\,dxdt\right|\\
\leq \left|\iint_{Q_\tau} f Du\cdot D\rho\,dxdt\right|+\left|\iint_{Q_\tau} f \mathrm{Tr}(D^2u)\rho\,dxdt\right|=:I_1+I_2
\end{multline*}
The term $I_1$ can be controlled by means of H\"older's and Young's inequalities, and the control on $\|\rho\|_{\mathcal{H}_{\bar q'}^1}$ stated in Corollary \ref{EstimateK2}:
\begin{multline}\label{bc3}
I_1\leq \|Du\|_{L^\infty(Q_\tau)}\|f\|_{L^{\bar{q}}(Q_\tau)}\|D\rho\|_{L^{\bar{q}'}(Q_\tau)}\leq C\|Du\|_{L^\infty(Q_\tau)}\|f\|_{L^{\bar{q}}(Q_\tau)} \big( \|Du\|^{1-\delta}_{L^\infty(Q_\tau)} + 1\big)\\ \leq C_3+\frac1{16}\|Du\|_{L^\infty(Q_\tau)}^2. 
\end{multline}
We apply to $I_2$ also H\"older's and Young's inequalities to get, for a $\bar p > 1$ to be chosen,
\begin{multline*}
I_2 
\leq \frac{1}{2\lambda}\iint_{Q_\tau} f^2\rho\,dxdt+\frac{\lambda}{2}\iint_{Q_\tau}\mathrm{Tr}((D^2u)^2)\rho\,dxdt\\ \le \frac{1}{2\lambda} \|f\|^2_{L^{2\bar{p}}(Q_\tau)}\|\rho\|_{L^{\bar{p}'}(Q_\tau)}+\frac{\lambda}{2}\iint_{Q_\tau}\mathrm{Tr}((D^2u)^2)\rho\,dxdt.
\end{multline*}
Let us focus on the first term of the right-hand side of the above inequality: it can be bounded by \eqref{estik4} and $\|f\|_{L^{q}(Q_T)}$ whenever there exists $\bar p$ such that
\[
  \frac{2(d+2)}{2(\gamma'-1) + 1} < 2\bar p \le q.
\]
Such a $\bar p$ indeed exists, since $q > \min\left\{d+2, \frac{d+2}{2(\gamma'-1)}\right\}$. Therefore,
\begin{equation}\label{bc4}
I_2 
\leq C_3+\frac1{16}\|Du\|_{L^\infty(Q_\tau)}^2 +\frac{\lambda}{2}\iint_{Q_\tau}\mathrm{Tr}((D^2u)^2)\rho\,dxdt.
\end{equation}

For the last term $\iint u_{x_k}\mathrm{Tr}(A_{x_k}D^2u)\rho$, Cauchy-Schwartz and Young's inequalities yield
\[
\iint_{Q_\tau} u_{x_k}\mathrm{Tr}(A_{x_k}D^2u)\rho\,dxdt\leq C \|Da\|^2_{\infty}\iint_{Q_\tau}|Du|^2\rho\,dxdt+\frac{\lambda}{2}\iint_{Q_\tau}\mathrm{Tr}((D^2u)^2)\rho\,dxdt
\]
We distinguish two cases: if $\gamma \ge 2$, we have by \eqref{estik1} (with $\beta = 2$) that $\iint_{Q_\tau}|Du|^2\rho \le C$. Otherwise, if $1 < \gamma < 2$, \[\iint_{Q_\tau}|Du|^2\rho \le \|Du\|_{L^\infty(Q_\tau)}^{2-\gamma} \iint_{Q_\tau}|Du|^\gamma \rho\,dxdt \le C \|Du\|_{L^\infty(Q_\tau)}^{2-\gamma}. \]
In both cases we end up with
\begin{equation}\label{bc5}
\iint_{Q_\tau} \partial_ku\mathrm{Tr}(\partial_kAD^2u)\rho\,dxdt\leq  C_4+\frac1{8}\|Du\|_{L^\infty(Q_\tau)}^2 +\frac{\lambda}{2}\iint_{Q_\tau}\mathrm{Tr}((D^2u)^2)\rho\,dxdt.
\end{equation}

{\bf Step 3.} Plugging \eqref{bc1}, \eqref{bc2}, \eqref{bc3}, \eqref{bc4} and \eqref{bc5} into \eqref{ineqBern} we get
\[
\frac12\int_{\T}|Du(x,\tau)|^2\rho_{\tau}(x)\,dx = \int_{\T}z(x,\tau)\rho_{\tau}(x)\,dx\leq \frac12 \|D u(\cdot, 0)\|_{L^\infty(\T)}^2+C+\frac3{8}\|Du\|_{L^\infty(Q_\tau)}^2.
\]
Since this inequality holds for all smooth $\rho_\tau \ge 0$ with $\|\rho_\tau\|_{L^1(\T)} = 1$, we obtain
\[
\frac12\|D u(\cdot, \tau)\|^2_{L^\infty(\T)}\leq \frac12 \|D u(\cdot, 0)\|^2_{L^\infty(\T)}+C+\frac3{8}\|Du\|_{L^\infty(Q_\tau)}^2,
\]
and we conclude by passing to the supremum with respect to $\tau \in (0,T)$.
\end{proof}

\appendix
\section{Some auxiliary results}

\begin{lemma}\label{parabolicreg}
Let $p>1$, $f \in L^p(Q_T)$, and suppose that $a_{ij}\in C(Q_T)$ satisfies \eqref{A1}. Then, there exists a unique solution in $W^{2,1}_p(Q_T)$ to
\begin{equation*}
\begin{cases}
\partial_tu(x,t)-\sum_{i,j=1}^da_{ij}(x,t)\partial_{ij}u(x,t)=f(x,t)&\text{ in }Q_T\ ,\\
u(x,0)=0&\text{ in }\T\ .
\end{cases}
\end{equation*}
 Moreover, there exists a constant $C$ (depending on $\lambda$, $p$, and the modulus of continuity of $a$ on $Q_T$) such that
\begin{equation}\label{estex}
\|u\|_{W^{2,1}_p(Q_T)}\leq C\|f\|_{L^p(Q_T)}\ .
\end{equation}
\end{lemma}

\begin{proof} This is a classical maximal $L^p$ regularity statement for uniformly elliptic equations with continuous coefficients, that can be deduced from results contained in \cite{LSU}; see \cite{CGM} for additional details. One can also rely on abstract results on maximal regularity for parabolic equations in \cite{PS}.
\end{proof}

The following continuous embedding result of $\mathcal{H}_\sigma^1(Q_T)$ into $L^p(Q_T)$ is rather known and can be found for example in \cite{CT}. However, we need its stability as $T \to 0$: this requires an additional control on the trace at some time (e.g. $t=0$). We provide a proof here for the reader's convenience.

\begin{prop}\label{embedding2}
If $1<\sigma<(d+2)/(d+1)$, then $\mathcal{H}_\sigma^1(Q_T)$ is continuously embedded into $L^p(Q_T)$ for
\[
\frac1p=\frac1\sigma -\frac1{d+2}.
\] 
Moreover, if $u \in \H_\sigma^1(Q_T)$ and $u(\cdot, 0) \in L^1(\T)$, we have
\begin{equation}\label{immL1}
\norm{u}_{L^p(Q_T)}\leq C\big(\norm{u}_{\H_\sigma^1(Q_T)}+\norm{u(0)}_{L^1(\T)}\big)\ ,
\end{equation}
where the constant $C$ depends on $d,p,\sigma,T$, but remains bounded for bounded values of $T$.
\end{prop}

\begin{proof} Let $f \in L^{p'}(Q_T)$ and $\varphi$ be the solution to 
\begin{equation*}
\begin{cases}
-\partial_t\varphi(x,t)-\Delta\varphi(x,t)=f(x,t)&\text{ in }Q_T\ ,\\
\varphi(x,T)=0&\text{ in }\T\ .
\end{cases}
\end{equation*}
By Lemma \ref{parabolicreg}, $\varphi$ satisfies 
\begin{equation}\label{ugabuga}
\|\varphi\|_{W^{2,1}_{p'}(Q_T)}\leq C\|f\|_{L^{p'}(Q_T)}\ .
\end{equation}
Note that $C$ here may depend on $T$, but it is the same for all $T \le 1$ (if $T < 1$, it is sufficient to extend trivially $f$ on $\T \times (T,1)$ and use \eqref{estex} on $\T \times (0,1)$). Note that $(d+2)/2 < p' < d+2$. Therefore, by the embedding results in \cite[Lemma II.3.3]{LSU},
\begin{equation}\label{embe99}
\|\varphi\|_{C(Q_T)}\leq C \|\varphi\|_{W^{2,1}_{p'}(Q_T)}, \qquad \|\varphi\|_{W^{1,0}_{\sigma'}(Q_T)}\leq C \|\varphi\|_{W^{2,1}_{p'}(Q_T)}
\end{equation}
Note that a straightforward application of \cite[Lemma II.3.3]{LSU} yields bounded constants in \eqref{embe99} as $T\to 0$, plus an additional term on the right-hand sides of the form $C_1 T^{-1} \|\varphi\|_{L^{p'}(Q_T)}$; this term can be removed using the fact that $\varphi(T)=0$, that guarantees $\|\varphi\|_{L^{p'}(Q_T)} \le T \|\partial_t \varphi\|_{L^{p'}(Q_T)} \le T \|\varphi\|_{W^{2,1}_{p'}(Q_T)}$. Note also that here we can identify norms on $\T$ with norms on $\Omega = (0,1)^d$.

Therefore, integrating by parts in time and using \eqref{ugabuga} and \eqref{embe99},
\begin{multline*}
\left| \iint_{Q_T} uf \, dxdt\right| = \left| \iint_{Q_T} u(-\partial_t\varphi-\Delta\varphi) \, dxdt\right| \\ \le \int_{\T}|\varphi(x,0)u(x,0)|dx + \left| \iint_{Q_T}\partial_t u\, \varphi \, dxdt  \right|+ \iint_{Q_T}|D\varphi|\,|Du|\,dxdt \\
\le C\Big(\|\varphi(0)\|_{L^\infty(\T)} \|u(0)\|_{L^1(\T)} + \|\partial_t u\|_{\big(W^{1,0}_{\sigma'}(Q_T)\big)'} \|\varphi\|_{W^{1,0}_{\sigma'}(Q_T)} + \|D u\|_{L^{\sigma}(Q_T)} \|D \varphi\|_{L^{\sigma'}(Q_T)}\Big) \\
\le C\Big(\|u(0)\|_{L^1(\T)} + \|\partial_t u\|_{\big(W^{1,0}_{\sigma'}(Q_T)\big)'} + \|D u\|_{L^{\sigma}(Q_T)}\Big)\|f\|_{L^{p'}(Q_T)},
\end{multline*}
yielding the desired result.

\end{proof}

\medskip
\begin{flushright}
\noindent \verb"marcoalessandro.cirant@unipr.it"\\
Dipartimento di Scienze Matematiche Fisiche e Informatiche\\ Universit\`a di Parma\\ Parco Area delle Scienze 53/a, 43124 Parma (Italy)

\noindent \verb"alessandro.goffi@gssi.it"\\
Gran Sasso Science Institute\\
viale Francesco Crispi 7, 67100 L'Aquila (Italy)
\end{flushright}

\end{document}